\documentclass[12pt]{amsart}
\usepackage[top=1in, bottom=1in, left=1in, right=1in]{geometry}

\usepackage{amssymb,amsmath,amsthm}
\usepackage{mathrsfs} 
\usepackage{bbm} 
\usepackage{enumitem}
\usepackage[hyphens]{url}
\usepackage{times}

\usepackage{graphicx}
\usepackage{color}
\usepackage[colorlinks=true]{hyperref}   

\numberwithin{equation}{section}

\newtheorem{thm}{Theorem}[section]
\newtheorem{lem}[thm]{Lemma}

\newtheorem{conj}[thm]{Conjecture}
\newtheorem{cor}[thm]{Corollary}

\theoremstyle{definition}

\theoremstyle{remark}
\newtheorem{remark}{Remark}[section]
\newtheorem*{remark*}{Remark}

\newtheoremstyle{claim} 
    {1em}                    
    {1em}                    
    {}                   
    {}                           
    {\bfseries}                   
    {.}                          
    {.5em}                       
    {}  
    
\theoremstyle{claim}

\newcommand{\m}{\mathfrak{m}}

\newcommand{\bs}\boldsymbol{}

\DeclareMathOperator{\sym}{sym}

\DeclareMathOperator{\sgn}{sgn}

\DeclareMathOperator{\SL}{SL}

\newcommand{\brac}[1]{\langle#1\rangle}
\renewcommand{\tilde}{\widehat}
\renewcommand{\tilde}{\widetilde}

\renewcommand{\phi}{\varphi}

\renewcommand{\Re}{{\rm Re}}
\renewcommand{\Im}{{\rm Im}}

\renewcommand{\bar}[1]{\overline{#1}}

\renewcommand{\mod}[1]{\,({\rm mod}\,#1)}

\definecolor{red}{rgb}{1,0,0}

\definecolor{orange}{rgb}{0.7,0.3,0}

\definecolor{blue}{rgb}{.2,.6,.75}

\definecolor{green}{rgb}{.4,.7,.4}

\begin{document}

\title{Quantum variance for holomorphic Hecke cusp forms on the vertical geodesic}

\author{Peter Zenz}
\email{peter.zenz@mail.mcgill.ca}

%

\begin{abstract}We compute the quantum variance of holomorphic cusp forms on the vertical geodesic for smooth, compactly supported test functions. The variance is related to an averaged shifted-convolution problem that we evaluate asymptotically. We encounter an off-diagonal term that matches exactly with a certain diagonal term, a feature reminiscent of moments of $L$-functions.
\end{abstract}

\maketitle


\section{Introduction}
The \textit{Quantum Unique Ergodicity} (QUE) conjecture, introduced by Rudnik and Sarnak \cite{rudnickBehaviourEigenstatesArithmetic1994}, is concerned about the equidistribution of Hecke eigenforms as the Laplace eigenvalue tends to infinity. More precisely, consider a smooth compactly supported function $\psi$ on $X:=SL_2(\mathbb{Z})\backslash \mathbb{H}$ and let $f(z)$ be a $L^2$-normalized Hecke Maass cusp form with Laplace eigenvalue $\lambda$. Then QUE predicts that
$$\int_{X} \psi(z)|f(z)|^2 \frac{dxdy}{y^2} \to \frac{3}{\pi} \int_{X} \psi(z) \frac{dxdy}{y^2}$$ as $\lambda \to \infty$. This conjecture is now a known theorem by the work of Lindenstrauss \cite{lindenstraussInvariantMeasuresArithmetic2006} and Soundararajan \cite{soundararajanQuantumUniqueErgodicity2010}. We can formulate an analogous conjecture for holomorphic Hecke cusp forms, where in the relation above we replace $f(z)$ with $f(z)y^{k/2}$ and the limit is the same, though this time it is $k$ that tends to infinity. This problem was solved by Holowinsky and Soundararajan in \cite{holowinskyMassEquidistributionHecke2010c}.

Going further, we can raise questions  about the distribution of eigenforms in certain subsets of the fundamental domain. In \cite{youngQuantumUniqueErgodicity2016}, Young considers several problems in that direction, such as QUE restricted to geodesics, horocycles or shrinking discs. Among other results, he showed that Eisenstein series satisfy QUE restricted to the vertical geodesic. In this paper we study the distribution of holomorphic Hecke cusp forms on the vertical geodesic. In this case, Young puts forward the following QUE conjecture:
\begin{conj}\cite[Conjecture 1.1]{youngQuantumUniqueErgodicity2016}\label{verticalgeo}
Suppose that $\psi\colon \mathbb{R}^+ \to \mathbb{R}$ is a smooth, compactly-supported function. Then
$$\lim_{k\to \infty} \int_0^\infty y^{k} |f(iy)|^2 \psi(y) \frac{dy}{y} = \frac{3}{\pi} \int_0^\infty \psi(y)\frac{dy}{y},$$
where $f(z)$ runs over weight $k$ holomorphic Hecke cusp forms that are $L^2$-normalized. 
\end{conj} 
Young provides evidence for Conjecture \ref{verticalgeo} by translating the problem to a (shifted) moment of $L$-functions and invoking the random matrix theory conjectures (\cite{conreyIntegralMomentsLfunctions2002}). Conjecture \ref{verticalgeo} is out of reach of current technology. As Young mentions, already a sharp upper bound $\int_0^\infty |f(iy)|^2y^k \frac{dy}{y} \ll k^\varepsilon$, would imply a subconvexity bound for $L(f, 1/2) \ll C^{1/8+\varepsilon}$, where $C\asymp k^2$ is the analytic conductor of $L(s, f)$. A bound of this strength is currently not even known for the simpler Riemann zeta function. 

As a first step toward Conjecture \ref{verticalgeo}, it is therefore natural to ask whether equidistribution holds for \textit{almost all} eigenforms, similar to the initial papers on equidistribution results for Hecke cusp forms by Luo and Sarnak (\cite{luoQuantumErgodicityEigenfunctions1995} and \cite{luoandsarnakMassEquidistributionHecke}).  The equidistribution problem for almost all eigenforms is often termed \textit{Quantum Ergodicity} (QE) in the literature. The Quantum Ergodicity problem restricted to certain hypersurfaces was studied in the case of Maass forms, or Laplacian eigenfunctions in general, by Christianson-Toth-Zelditch in \cite{christiansonQuantumErgodicRestriction2012}, Dyatlov-Zworski in \cite{dyatlovQuantumErgodicityRestrictions2012} and Toth-Zelditch in \cite{tothQuantumErgodicRestriction2011}.  We will show such a QE result for holomorphic Hecke cusp forms on the vertical geodesic. In fact, it will be a direct consequence of the stronger \textit{quantum variance} result we will describe now:

Let $f$ be a holomorphic Hecke cusp form of weight $k$ for the full modular group $\Gamma=\SL_2(\mathbb{Z})$ on the upper half plane $\mathbb{H}$. We normalize $f$ such that $||f||_2^2=\brac{f,f}=1$, where
$$\brac{f,g}:=\int_{X}f(z)\overline{g(z)}y^{k}\frac{dxdy}{y^2}$$
is the standard Petersson inner product for two holomorphic cusp forms $f,g$ of weight $k$. For a smooth, compactly supported test function $\psi\colon \mathbb{R^+} \to \mathbb{R}$, we define
\begin{equation}\label{vertquant}
\mu_f(\psi)=\int_0^\infty |f(iy)|^2 y^{k} \psi(y)\frac{dy}{y} \quad \text{and} \quad \mathbb{E}(\psi)=\frac{3}{\pi}\int_0^\infty \psi(y)\frac{dy}{y}.
\end{equation}
 To measure how far $\mu_f(\psi)$ deviates from its expected value $\mathbb{E}(\psi)$ we introduce the so-called \textit{quantum variance} (for the vertical geodesic), which is (up to suitable normalization) given by
$$\sum_{K < k \leq 2K} \sum_{f\in H_k} \big| \mu_f(\psi)-\mathbb{E}(\psi)\big|^2,$$
where $H_k$ is a basis of Hecke cusp forms. Note, that we will work with even weight functions, i.e. $\psi(y)=\psi(1/y)$, since we can decompose a test function into its even and odd part and the quantum variance for an odd function is identically $0$. In particular, we also have $\tilde{\psi}(s)=\tilde{\psi}(-s)$, where
 $$\tilde{\psi}(s):=\int_0^\infty \psi(y^{-1})y^{s-1} dy.$$

The quantum variance problem was first investigated in a different setting by Zelditch in \cite{zelditchRateQuantumErgodicity1994}. Luo and Sarnak were the first to compute the asymptotic of the quantum variance for holomorphic Hecke cusp forms on the full fundamental domain (see \cite{luoQuantumVarianceHecke2004a}). Since then, several variants of this problems were explored. Zhao \cite{zhaoQuantumVarianceMaassHecke2010} considered the problem for Hecke Mass cusp forms and Sarnak-Zhao \cite{sarnakQuantumVarianceModular2013} further explored it for the modular surface. Closed geodesics on the modular surface were investigated my Luo, Rudnick and Sarnak in \cite{luoVarianceArithmeticMeasures2009}.  In the compact setting of quaternion algebras  Nelson \cite{nelsonQuantumVarianceQuaternion2016},\cite{nelsonQuantumVarianceQuaternion2017}, \cite{nelsonQuantumVarianceQuaternion2019} evaluated the quantum variance using the theta correspondence. Further papers indicating the active field of research are for example given by work of Huang \cite{huangQuantumVarianceEisenstein2021} for the variance of Eisenstein series, Huang-Lester \cite{huangQuantumVarianceDihedral2020} for the variance of dihedral Maass cusp forms and the work of Nordentoft, Petridis and Risager \cite{nordentoftSmallScaleEquidistribution2021b} on small scale equidistribution at infinity.

We explore for the first time the one-dimensional case of the vertical geodesic for holomorphic cusp forms and show
\begin{thm}\label{mainthm}Let $\psi_1,\psi_2$ and $h$ be  smooth compactly-supported functions on $\mathbb{R}^+$. Moreover, suppose that $\psi_i(y)=\psi_i(1/y)$ for $i=1,2$. Then 
\begin{equation}\label{thmequ1}\sum_{k\equiv 0\mod{2}} h\Big(\frac{k-1}{K}\Big) \sum_{f\in H_k} L(1, \sym^2 f)  \Big( \mu_f(\psi_1) - \mathbb{E}(\psi_1)\Big)\cdot \Big(\mu_f(\psi_2) - \mathbb{E}(\psi_2)\Big) = V(\psi_1, \psi_2)
\end{equation}
with 
\begin{align*}V(\psi_1, \psi_2)=&
K^{3/2}\log K\cdot \frac{\sqrt{2}\pi}{32}\tilde{\psi_1}(0)\tilde{\psi_2}(0)  \cdot \int_0^\infty \frac{h(\sqrt{u})u^{1/4}}{\sqrt{2\pi u}} du +\\
&+K^{3/2} \frac{\sqrt{2}\pi}{64}\tilde{\psi_1}(0)\tilde{\psi_2}(0)\int_0^\infty \frac{h(\sqrt{u})u^{1/4}}{\sqrt{2\pi u}} \log(u) du+\\
&+K^{3/2}\int_0^\infty \frac{h(\sqrt{u})u^{1/4}}{\sqrt{2\pi u}} du \cdot \Big(\frac{\sqrt{2}\pi}{16}\Big(\frac{3}{2} \gamma - \log (4\pi)\Big) \tilde{\psi_1}(0)\tilde{\psi_2}(0) \Big)+\\
&+K^{3/2} \int_0^\infty \frac{h(\sqrt{u})u^{1/4}}{\sqrt{2\pi u}} du \cdot \frac{\sqrt{2}\pi}{8} \frac{1}{2\pi i} \int_{(1)} \tilde{\psi}_1(-s_2)\tilde{\psi_2}(s_2)\zeta(1-s_2)\zeta(1+s_2)ds_2+ \\
&+O_{\psi_1, \psi_2}(K^{5/4+\varepsilon})
\end{align*}
as $K \to \infty$. 
\end{thm}
In the computation of quantum variances one usually chooses a test function $\psi$ that is on average $0$, i.e. $\mathbb{E}(\psi)=0$. As a corollary of Theorem \ref{mainthm} we then get
\begin{cor}\label{maincor} Let $\psi_1,\psi_2$ and $h$ be  smooth compactly-supported functions on $\mathbb{R}^+$. Moreover, suppose that $\psi_i(y)=\psi_i(1/y)$ and $\mathbb{E}(\psi_i)=0$ for $i=1,2$, then 
\begin{equation}\label{thmequ2}
\sum_{k\equiv 0\mod{2}} h\Big(\frac{k-1}{K}\Big) \sum_{f\in H_k} L(1, \sym^2 f)   \mu_f(\psi_1) \mu_f(\psi_2) = V(\psi_1, \psi_2)
\end{equation}
with
\begin{align}\label{slickvar}
V(\psi_1, \psi_2)=&K^{3/2} \int_0^\infty \frac{h(\sqrt{u})u^{1/4}}{\sqrt{2\pi u}} du \cdot \frac{\sqrt{2}\pi}{8} \frac{1}{2\pi i} \int_{(1)} \tilde{\psi_1}(-s_2)\tilde{\psi_2}(s_2)\zeta(1-s_2)\zeta(1+s_2)ds_2 \\&+ O_{\psi_1, \psi_2}(K^{5/4+\varepsilon}). \nonumber
\end{align}
as $K\to \infty$.
\end{cor}


As mentioned before, a direct consequence of Theorem \ref{mainthm} is the following equidistribution result:
\begin{cor}[Quantum Ergodicity]
Let $f$ be a holomorphic Hecke  cusp form of even weight $k$ and $\psi$ a real-valued, smooth, compactly-supported on $\mathbb{R}^+$. For a fixed $\epsilon>0$ we have
$$\#\big\{f \in \bigcup_{\substack{K <  k \leq 2K\\ k\equiv 0 \mod{2}}} H_k: |\mu_f(\psi)-\mathbb{E}(\psi)| > K^{-1/4+\epsilon}\big\} =o(K^2),$$
i.e. almost all eigenforms $f \in  \bigcup_{\substack{K <  k \leq 2K\\ k\equiv 0 \mod{2}}} H_k$ are equidistributed.
\end{cor}

In subsequent work we will also address an application of Theorem \ref{mainthm} to count real zeros (for almost all) holomorphic Hecke cusp forms low in the fundamental domain (say with $y$ in the interval $[1,2]$). Gosh and Sarnak \cite{ghoshRealZerosHolomorphic2011c} considered  this problem for all Hecke cusp forms high in the fundamental domain (close to the cusp).
\section{Acknowledgements}
The author would like to thank Maksym Radziwi\l\l~and Dimitris Koukoulopoulos for their continuous support during this project and many helpful conversations. Moreover, the author would like to thank Jake Chinis, Peter Humphries, Zeev Rudnick and Steve Zelditch for their interest and comments as well as Junehyuk Jung for suggesting further relevant references. 

\section{Comparing the full fundamental domain with the vertical geodesic}
We now highlight some of the differences between the variance computation for the full fundamental domain (denoted by $X$) and the vertical geodesic. To do so we revisit the work of Luo and Sarnak, who compute the quantum variance for the full fundamental domain in \cite{luoQuantumVarianceHecke2004a}. Set
$$\mu_f^{X}(\psi):=\int_{X} \psi(z) |f(z)|^2 y^{k} \frac{dxdy}{y^2} \quad \mathbb{E}^{X}(\psi):=\int_{X}\psi(z) \frac{dxdy}{y^2},$$
 $\psi \colon X \to \mathbb{R}$ is a smooth, compactly supported function and $z=x+iy$. We further restrict our attention to test functions $\psi$ that have mean value $0$, i.e. $\mathbb{E}^{X}(\psi)=0$. As in \cite{luoQuantumVarianceHecke2004a} we denote by
$$V^{X}(\psi)= \sum_{K < k \leq 2K} \sum_{f\in H_k} \big| \mu_f^{X}(\psi)\big|^2$$
the quantum variance for the full fundamental domain. Computing the variance $V^{X}(\psi)$ reduces to computing an averaged shifted convolution problem of the form
$$\mathcal{M}^X(h_{\psi}) = \sum_{K < k \leq 2K} \sum_{f\in H_k} \Big| \frac{1}{k} \sum_{n} \lambda_f(n)\lambda_f(n+\ell)h_{\psi}\Big(\frac{k}{n}\Big)\Big|^2,$$
in the sense that an asymptotic (or bound) of $\mathcal{M}^X(h_{\psi})$ will lead to an asymptotic (or bound) of $V^{X}(\psi)$ (see \cite[Theorem 2 and Section 4]{luoQuantumVarianceHecke2004a}). In the definition of $\mathcal{M}^X(h_{\psi})$ the shift $\ell$ is a fixed non-zero integer, $\lambda_f(n)$ denotes the $n$-th Hecke eigenvalue of $f$, and $h_{\psi}$ is a smooth, compactly supported test function (that arises after decomposing $\psi$ into Poincar\'{e} series) which forces $n$ to be of size $k$.
We expect $\mathcal{M}^X(h_{\psi}) =O(K^{1+\epsilon})$, assuming square root cancellation in the shifted convolution problem. Indeed, Luo and Sarnak \cite[Theorem 1, Theorem 2]{luoQuantumVarianceHecke2004a} compute the asymptotic for $\mathcal{M}^X(h_{\psi})$, and show that $\mathcal{M}^X(h_{\psi})$  and $V^{X}(\psi)$ are of size $K$. More precisely, they obtain the asymptotic 

$$\mathcal{M}^X(h_{\psi})\sim B(h_{\psi}) K $$
with
\begin{align*} B(h_{\psi})=&\frac{\pi}{4} \sum_{d_1|\ell_1, d_2|\ell_2; |\ell_1|/d_1=|\ell_2|/d_2=|\ell_2|/d_2} \frac{1}{d_1d_2}\int_0^\infty h_{\psi}(d_2\eta)\overline{h_{\psi}}(d_1\eta)\frac{d\eta}{\eta^2} \\
&-\frac{\pi}{2\sqrt{2}} \sum_{d_1|\ell_1, d_2|\ell_2} \frac{1}{d_1d_2} \sum_{c\geq 1} \frac{S_{c, |\ell_1|/d_1, \ell_2/d_2}}{c^{5/2}} e\Big(\frac{1}{2c} \frac{|\ell_1|}{d_1}\frac{|\ell_2|}{d_2}\Big) \\
&\times \int_{-\infty}^\infty \int_{-\infty}^\infty \sin\Big(-\frac{\pi}{4}-\frac{\pi}{2c}\Big(\frac{\ell_1}{d_1}\Big)^2 \frac{\xi}{\eta}-\frac{\pi}{2c}\Big(\frac{\ell_2}{d_2}\Big)^2 \frac{\eta}{\xi} + 2\pi(d_1d_2)^2\xi \eta c\Big) \\
&\times \frac{h_{\psi}(d_2\xi)}{\sqrt{\xi}} \frac{h_{\psi}(d_1\eta)}{\sqrt{\eta}} \frac{d\xi d\eta}{\xi \eta},
\end{align*}
where 
\begin{equation}\label{Salie}S_{c, \ell_1/d_1, \ell_2/d_2}:= \sum_{a, b \mod{c}}S(a(a+\ell_1/d_1), b(b+\ell_2/d_2);c)e_c\Big(2ab +\Big(\frac{\ell_2}{d_2}\Big)a+\Big(\frac{\ell_1}{d_1}\Big) b\Big)
\end{equation} and $S(m, m; c)$ denotes the classical Kloosterman sum and $e_c(x)=e^{2\pi i x/c}$.

The quantum variance for the vertical geodesic (with $\mathbb{E}(\psi)=0$) is given by
$$V(\psi)=\sum_{K< k\leq 2K}\sum_{f\in H_k}\big | \mu_f(\psi)\big|^2.$$
After a Fourier expansion of the Hecke cusp form $f$ and opening the square we are facing a different shifted convolution problem:

$$\mathcal{M}(\psi) = \sum_{K < k \leq 2K} \sum_{f \in H_k} \Big| \frac{1}{k} \sum_{0  < |\ell|\leq k^{1/2+\epsilon}} \sum_{n} \lambda_f(n)\lambda_f(n+\ell)\psi\Big(\frac{n}{k}\Big)\Big|^2.$$
We notice an additional summation over the shifts $\ell$ up to $\sqrt{k}$, compared to the full fundamental domain where shifts are fixed. Again, assuming square root cancellation in the summation over $n$ and $\ell$  leads us to believe that $\mathcal{M}(\psi)=O(K^{3/2})$.

Although, the additional summation over $\ell$ makes the problem of computing the variance more difficult (as we need cancellation over the summation of $n$ and $\ell$), it also gives rise to a surprising feature of the vertical geodesic, namely the aesthetically pleasing formula
\begin{align}\label{MainShifted}
\mathcal{M}(\psi)\sim&B(\psi) K^{3/2}
\end{align}
with $$B(\psi)=\frac{\sqrt{2}\pi}{8} \frac{1}{2\pi i} \int_{(1)} \tilde{\psi}(-s_2)\tilde{\psi}(s_2)\zeta(1-s_2)\zeta(1+s_2)ds_2.$$
In the following section we quickly sketch the proof of the asymptotic formula \eqref{MainShifted} (and hence Theorem \ref{mainthm}).  

\section{High Level Sketch}
As mentioned before, to compute the quantum variance it suffices to consider the shifted convolution problem
\begin{equation}\label{heuristic}
\mathcal{M}(\psi)= \sum_{K < k \leq 2K} \sum_{f\in H_k} \Big|\frac{1}{k}\sum_{0<|\ell| \leq k^{1/2+\varepsilon}}\sum_{n} \lambda_f(n)\lambda_f(n+\ell)\psi\Big(\frac{k}{n}\Big)\Big|^2.
\end{equation}
Our goal is to show that $\mathcal{M}(\psi)$ is of size $K^{3/2}$. If we can show that $\mathcal{M}(\psi)=o(K^2)$ then we could deduce an equidistribution result, i.e. that $\mu_f(\psi)=\mathbb{E}(\psi)+o(1)$ for almost all Hecke cusp forms in $\bigcup_{\substack{K <  k \leq 2K\\ k\equiv 0 \mod{2}}} H_k$. Observe that the trivial bound for $\mathcal{M}(\psi)$ is $K^{3+\epsilon}$, since the Fourier coefficients $\lambda_f(n)$ are divisor bounded. On the other hand, assuming square root cancellation only in the summation over $n$ yields the bound $\mathcal{M}(\psi)=O(K^{2+\epsilon})$. We therefore need to detect further cancellation over the shifts $\ell$ to achieve our desired equidistribution result.

We now explain how to compute the asymptotic formula for $\mathcal{M}(\psi)$. First, we open the square in \eqref{heuristic}  and apply the Petersson trace formula. We obtain a diagonal term that is easy to evaluate and a more complicated off-diagonal expression involving Kloosterman sums. Let $\mathcal{OD}$ denote this off-diagonal term that is roughly given by
\begin{equation}\label{offdiagonalheuristic}\frac{1}{K}\sum_{0<|\ell_1|, |\ell_2|\leq K^{1/2+\varepsilon}} \sum_{K < n_1, n_2 \leq 2K} \sum_{c\ll K^\varepsilon} \frac{S(n_1(n_1+\ell_1), n_2(n_2+\ell_2);c)}{\sqrt{c}} e_c(2\sqrt{n_1(n_1+\ell_1)n_2(n_2+\ell_2)}),
\end{equation}
where $S(n,m;c)$ represents the classical Kloosterman sum and $e_c(n)=e^{2\pi i n/c}$. At this point, a trivial bound (using the Weil bound for Kloosterman sums) shows that $\mathcal{OD} = O(K^{2+\epsilon})$. In particular, we need to exploit a little bit more cancellation over the shifts to break the barrier for equidistribution.

As the summation over the variable $c$ is short, it is instructive to consider the special case when $c=1$ so that 
$$\mathcal{OD}\approx \frac{1}{K}\sum_{0<|\ell_1|, |\ell_2|\leq K^{1/2+\varepsilon}} \sum_{K < n_1, n_2 \leq 2K} e_1(2\sqrt{n_1(n_1+\ell_1)n_2(n_2+\ell_2)}).$$
The difference to the full fundamental domain, as noted before, is given by the size of the shifts. When the shifts are fixed, the exponential $e_c(2\sqrt{n_1(n_1+\ell_1)n_2(n_2+\ell_2)})$ is essentially smooth. For the vertical geodesic this is not the case and we expect cancellation when summing over the various variables $n_i, \ell_i$ for $i=1,2$. We will sum at first over the variable $n_1$ and detect square root cancellation in 
$$\sum_{n_1 \sim K}e_c(2\sqrt{n_1(n_1+\ell_1)n_2(n_2+\ell_2)})$$
with a stationary phase argument. The new phase in the exponential collapses to zero and so there is no further cancellation to be found. This will show that $\mathcal{OD}=O(K^{3/2+\epsilon})$, as desired. To compute the exact asymptotic of $\mathcal{OD}$ we use the long shifts $\ell_1, \ell_2$ to our advantage. We split the various variables into residue classes modulo $c$, namely $n_i\equiv a_i \mod{c}$ and $\ell_i\equiv b_i\mod{c}$ for $i=1,2$. After the stationary phase argument from before we need to evaluate
\begin{equation}\label{klooster}
S:=\sum_{\substack{a_1\mod{c}\\a_2\mod{c}}} \sum_{\substack{b_1 \mod{c}\\b_2 \mod{c}}} S_c(a_1(a_1+b_1), a_2(a_2+b_2); c)e_c(2a_1a_2+a_1b_2+a_2b_1).
\end{equation}
 The additional summations over the parameters $b_1, b_2$ arising from the long shifts $\ell_1, \ell_2$ are crucial to show that $S=c^3\phi(c)$. Here $\phi(c)$ denotes as usual Euler's totient function. Let us compare quantity $S$ with expression \eqref{Salie} for the full fundamental domain. Expression \eqref{Salie} does not simplify significantly, but is related to Sali\'{e} sums instead. With the formula $S=c^3\phi(c)$ in hand and some complex contour integrals we can evaluate the off-diagonal $\mathcal{OD}$ asymptotically. The final  computations remind us of computing moments of $L$-functions, in particular, the paper \cite{khanNonvanishingSymmetricSquare2010} of Khan and the computation of the second moment of the symmetric square $L$-function therein. We also observe that the off-diagonal term agrees with a certain diagonal term, which enables us to write down the rather clean formula \eqref{slickvar} for the quantum variance.
\section{Proof of Theorem \ref{mainthm}}
\subsection{Setup}

Recall from the introduction that $\psi$, $h$ are smooth compactly-supported functions on $\mathbb{R^+}$ and $H_k$ denotes a basis of Hecke cusp forms of weight $k$. We want to compute the asymptotic of
\begin{equation*}\label{startingpoint}V(\psi_1, \psi_2)=\sum_{k\equiv 0\mod{2}} h\Big(\frac{k-1}{K}\Big) \sum_{f\in H_k} L(1, \sym^2 f)  \Big( \mu_f(\psi_1) - \mathbb{E}(\psi_1)\Big)\cdot \Big(\mu_f(\psi_2) - \mathbb{E}(\psi_2)\Big)
\end{equation*}
with
$$\mu_f(\psi)=\int_0^\infty |f(iy)|^2 y^{k/2} \psi(y) dy \quad \text{ and } \quad \mathbb{E}(\psi)=\frac{3}{\pi} \int_0^\infty \psi(y)\frac{dy}{y}.$$ Here the symmetric square L-function is given by
$$L(s, \sym^2f)=\zeta(2s)\sum_{n=1}^\infty \frac{\lambda_f(n^2)}{n^s},$$
for $\Re(s)>1$. The Fourier expansion of a normalized Hecke cusp form is given by
$$f(z)=a_f(1)\sum_{n=1}^\infty \lambda_f(n) (4\pi n)^{(k-1)/2} e(nz),$$
with $|a_f(1)|^2=\frac{2\pi^2}{\Gamma(k)L(1,\sym^2f)}$ arising from the normalization $||f||_2^2=1$. 
Similar to Luo and Sarnak in \cite[p. 877]{luoandsarnakMassEquidistributionHecke} we define the function $\tilde{\psi}(s)$ by
$$\tilde{\psi}(s):=\int_0^\infty \psi(y^{-1})y^{s-1} dy.$$ 
Then $\tilde{\psi}$ is entire, and for any integer $j>0$ and any vertical strip $a \leq \Re(s) \leq b$, it satisfies
$\tilde{\psi}\ll_{a, b, j} (|s|+1)^{-j}$. Mellin inversion yields
$$\psi(y)=\frac{1}{2\pi i} \int_{(\sigma)} \tilde{\psi}(s)y^s ds, \quad \text{for }\sigma >0, y >0.$$
\subsection{Reduction to a Shifted Convolution Problem}
To evaluate the variance $V(\psi_1, \psi_2)$ we first use the Fourier expansion of $f(iy)$ and write
$$|f(iy)|^2=\big|a_f(1)\sum_{n} \lambda_f(n)(4\pi n)^{(k-1)/2}e^{-2\pi n y} \big|^2.$$ 
We then expand the square, seperating the terms with $m=n$ from those with $m\neq n$. The terms with $m=n$ agree up to a small error term (which we will call $\mathcal{E}_\psi$ below) with the expected main term $\mathbb{E}(\psi)$. The terms with $m\neq n$ lead to a shifted convolution problem and this quantity will be denoted by $S_\psi$.
\begin{lem}We have
\begin{align}\mathcal{E}_\psi:=&\int_0^\infty \psi(y)y^k |a_f(1)|^2 \sum_{n=1}^\infty \lambda_f(n)^2 (4\pi n)^{k-1} e^{-4\pi ny} \frac{dy}{y} -\frac{3}{\pi}\int_0^\infty \psi(y)\frac{dy}{y}\label{errorterm}\\ 
=&\frac{2\pi^2}{L(1,\sym^2f)} \cdot{\frac{1}{2\pi i}} \int_{(1/2)}\tilde{\psi}(s-1) \frac{\zeta(s) L(s, \sym^2f)}{(4\pi)^{s} \zeta(2s)} \frac{\Gamma(k+s-1)}{\Gamma(k)} ds. \nonumber
\end{align}
\end{lem}
\begin{proof}
\begin{align*}&\int_0^\infty \psi(y)y^k |a_f(1)|^2 \sum_{n=1}^\infty \lambda_f(n)^2 (4\pi n)^{k-1} e^{-4\pi ny} \frac{dy}{y} \\
=&\frac{2\pi^2}{\Gamma(k)L(1, \sym^2f)} \cdot \frac{1}{2\pi i}\int_{(2)} \tilde{\psi}(s) \sum_{n=1}^\infty \lambda_f(n)^2 (4\pi n)^{k-1} \int_0^\infty y^{k+s} e^{-4\pi ny}\frac{dy}{y}\\
=&\frac{2\pi^2}{L(1,\sym^2f)} \cdot{\frac{1}{2\pi i}} \int_{(2)}\tilde{\psi}(s) \frac{\zeta(1+s)L(1+s, \sym^2f)}{(4\pi)^{1+s}\zeta(2(1+s))} \frac{\Gamma(k+s)}{\Gamma(k)} ds
\end{align*}
We used that we can write the Rankin-Selberg L-function in terms of the symmetric square as $$L(s, f\otimes f) =\frac{\zeta(s)L(s,\sym^2f)}{\zeta(2s)}= \sum_{n=1}^\infty \frac{\lambda_f(n)^2}{n^s}.$$
We then shift the contour from $\Re(s)=2$ to $\Re(s)=-1/2$ and pick up a pole at $s=0$ with residue $$\frac{2\pi}{L(1, \sym^2 f)}\cdot \tilde{\psi}(0) \frac{L(1, \sym^2 f)}{4\pi \zeta(2)} = \frac{3}{\pi} \int_0^\infty \psi(y) \frac{dy}{y}.$$
The lemma follows by making the change of variables $s \to s-1$ for the new line integral.
\end{proof}
The term $\mathcal{E}_{\psi}$ should be seen as an error term that is of size $K^{-1/2}$ (compare for example with \cite[Section 5]{luoandsarnakMassEquidistributionHecke}). The off-diagonal term on the other hand is given by a shifted convolution (which we denote by $S_\psi$) of size $K^{-1/4}$ as the following lemma indicates:
\begin{lem}\label{shiftedconvolutionint} We have
\begin{align}S_{\psi}:=& \int_0^\infty \psi(y) y^k |a_f(1)|^2 \sum_{n\neq m} \lambda_f(n)\lambda_f(m) (16\pi^2 nm)^{(k-1)/2}e^{-2\pi (n+m)y} \frac{dy}{y}\label{shiftedconvolutiondefi} \\
=&\frac{\pi}{2 L(1, \sym^2f)}\sum_{\ell\neq 0} \sum_n \frac{\lambda_f(n)\lambda_f(n+\ell)}{\sqrt{n(n+\ell)}} \exp\Big( -\frac{k\ell^2}{2(2n+\ell)^2}\Big) \psi\Big(\frac{k}{2\pi(2n+\ell)}\Big) +O_\psi(k^{-1/2+\varepsilon}).\nonumber
\end{align}
\end{lem}
\begin{remark}
Since $\psi$ is smooth compactly-supported on $\mathbb{R}^+$, we see that $(2n+\ell) \asymp k$. The exponential factor limits the size of $\ell$ as otherwise we have rapid decay. It follows that $\ell \ll k^{1/2+\varepsilon}$ and consequently $n \asymp k$. These observations show that the off-diagonal term is related to a shifted convolution problem of the form
$$\frac{1}{k}\sum_{0<|\ell|\ll k^{1/2+\varepsilon}} \sum_{n\sim k} \lambda_f(n)\lambda_f(n+\ell).$$
Assuming square root cancellation in $n$ and the shifts $\ell$ the expected size of $S_\psi$ is $k^{-1/4+\varepsilon}$. The main part of the paper is attributed to showing this statement on average.
\end{remark}
\begin{proof}[Proof of Lemma \ref{shiftedconvolutionint}]
First, we show by elementary means that only the terms satisfying $n+m \asymp k$ contribute to the main term of $S_\psi$.
Since $\psi(y)$ is compactly-supported, there exist real numbers $0<a<b$ such that $\psi(y)$ is supported in $[a,b]$. We write
$L(y)=-2\pi(n+m)y+k \log y$ and set $y_0=k/2\pi(n+m)$ so that $L(y)$ attains its maximum at $y_0$. Suppose that $y_0 \leq a/2$. Then for all $y$ in the support of $\psi$ we have 
$$L'(y) = -2\pi(n+m) +k/y \leq -2\pi(n+m) + k/(2y_0)\leq -\pi (n+m).$$ 
Hence,
\begin{align*}
\int_0^\infty \psi(y) e^{L(y)} \frac{dy}{y} &\leq \int_a^\infty \psi(y) e^{L(y_0)}e^{-(y-y_0) \pi (n+m)} \frac{dy}{y}\\
&\ll_\psi e^{L(y_0)} e^{-(a-y_0) \pi(m+n)}\\
&\ll_\psi e^{L(y_0)} e^{-a \pi(m+n)/2}
\end{align*}
It follows that the contribution from $n, m$ such that $k/(2\pi(n+m))\leq a/2$ to $S_\psi$ is bounded by
\begin{equation}\label{smally}\frac{2\pi^2}{L(1, \sym^2f)} \sum_{\substack{n, m \\ (n+m) \geq k/(a\pi)}} \frac{d(n)d(m)}{\sqrt{nm}} \cdot \Big(\frac{2\sqrt{nm}}{n+m}\Big)^k e^{-a\pi(n+m)/2}.
\end{equation}
Here we used the Deligne bound $|\lambda_f(n)|\leq d(n)$, where $d(n)$ is the divisor function. 
Expression \eqref{smally} decays exponentially in $k$ and is thus negligible.
Similarly, we treat the case when $y_0 \geq 2b$. Note that $L''(y)=-k/y^2 \leq -k/b^2$ for every $y$ in the support of $\psi$. Moreover, $(y-y_0)^2 \geq b^2$. We then have
\begin{align*}\int_0^\infty \psi(y)e^{L(y)} \frac{dy}{y}&\leq \int_0^\infty e^{L(y_0)-(y-y_0)^2 \frac{k}{b^2}} \frac{dy}{y} \\
&\ll_\psi e^{L(y_0)} e^{-k}.
\end{align*}
The contribution from $n, m$ such that $k/(2\pi(n+m))\geq 2b$ to $S_\psi$ is thus bounded by
\begin{equation}\frac{2\pi^2}{L(1, \sym^2f)} \sum_{\substack{n, m \\ (n+m) \leq k/(4\pi b)}} \frac{d(n)d(m)}{\sqrt{nm}} \cdot \Big(\frac{2\sqrt{nm}}{n+m}\Big)^k e^{-k},
\end{equation}
which decays exponentially in $k$.

Subsequently, we restrict our attention to the case $a/2 \leq y_0 \leq 2b$ and in particular, $n+m \asymp k$. 
We start by performing an inverse Mellin transform on $\psi$ and evaluating the integral over $y$ as a Gamma function:
\begin{align*}&\int_0^\infty \psi(y) y^k |a_f(1)|^2 \sum_{n\neq m} \lambda_f(n)\lambda_f(m) (16\pi^2 nm)^{(k-1)/2}e^{-2\pi (n+m)y} \frac{dy}{y} \\
=&\frac{2\pi^2}{\Gamma(k) L(1, \sym^2f)}\sum_{n\neq m} \lambda_f(n)\lambda_f(m) (16\pi^2nm)^{(k-1)/2} \frac{1}{2\pi i}\int_{(2)} \tilde{\psi}(s) \int_0^\infty y^{k+s} e^{-2\pi( n+m)y}\frac{dy}{y}ds \\
=&\frac{2\pi^2}{L(1, \sym^2f)}\sum_{n\neq m} \lambda_f(n)\lambda_f(m)(16\pi^2nm)^{(k-1)/2} \frac{1}{2\pi i}\int_{(2)} \tilde{\psi}(s) \frac{1}{\big(2\pi (n+m)\big)^{k+s}}\frac{\Gamma(k+s)}{\Gamma(k)}ds 
\end{align*}
Similar to \cite[Eq. 2.3]{luoandsarnakMassEquidistributionHecke} Stirlings formula yields that for any vertical strip $0<a \leq \Re(s)\leq b$,
\begin{equation}\label{Gamma}\frac{\Gamma(k+s)}{\Gamma(k)} = k^s \cdot (1+ O_{a,b}((1+|s|)^2 k^{-1})).
\end{equation}
Using \eqref{Gamma} we have
\begin{align}\label{shiftedconvolution}S_\psi=&\frac{\pi}{2L(1, \sym^2f)} \sum_{n\neq m} \frac{\lambda_f(n)\lambda_f(m)}{\sqrt{nm}} \cdot \Big(\frac{2\sqrt{nm}}{n+m}\Big)^k \psi\Big(\frac{k}{2\pi (n+m)}\Big) \\
&+O_\psi\Big(\frac{1}{k \cdot L(1, \sym^2f)} \sum_{n\neq m} \frac{d(n)d(m)}{\sqrt{nm}} \cdot \Big(\frac{2\sqrt{nm}}{n+m}\Big)^k \Big) \nonumber
\end{align}
The factor $\big(\frac{2\sqrt{mn}}{n+m}\big)^k$ is forcing $m$ and $n$ to be close (roughly $|m-n|\ll k^{1/2+\varepsilon}$). More precisely, note that 
\begin{align*}
\frac{m+n}{2\sqrt{mn}}&=\sqrt{\frac{(m+n)^2}{4mn}} \geq \sqrt{1+\frac{(m-n)^2}{4(m+n)^2}}
\end{align*}
In particular,
$$\Big(\frac{2\sqrt{mn}}{n+m}\Big)^k \leq e^{-O(|m-n|^2/k)}$$
and the contribution from $m, n$ with $|m-n|\geq k^{1/2+\varepsilon}$ is exponentially small in $k$. When $|m-n|\ll k^{1/2+\varepsilon}$ we have as in \cite[p. 9]{blomerDistributionMassHolomorphic2013b}
\begin{align*}
\Big(\frac{2\sqrt{mn}}{m+n}\Big)^k &= \Big( 1 - \frac{|m-n|^2}{2(m+n)^2} + O\Big( \frac{|m-n|^4}{(m+n)^4}\Big)\Big)^k\\
&=\exp\Big(k \log \Big( 1 - \frac{|m-n|^2}{2(m+n)^2} + O\Big( \frac{|m-n|^4}{(m+n)^4}\Big)\Big)
\end{align*}
By a Taylor expansion it follows that 
$$S_\psi = \frac{\pi}{2 L(1, \sym^2f)}\sum_{m\neq n} \frac{\lambda_f(m)\lambda_f(n)}{\sqrt{mn}}\exp\Big(-\frac{k|m-n|^2}{2(m+n)^2}\Big)\psi\Big(\frac{k}{2\pi (n+m)}\Big) +O_\psi(k^{-1/2+\varepsilon}).$$
The lemma follows upon writing $m=n+\ell$.
\end{proof}
Subsequently, we only consider the case when $m>n$ and thus $\ell>0$ as the case with $m<n$ is exactly the same upon relabelling the variables. 
\subsection{Cancellation in the Shifted Convolution Problem}
Our goal now is to detect cancellation in the shifted convolution sum $S_\psi$ (in an $L^2$ sense, when averaged over $k$ and $f\in H_k$). To do this we will use an averaged Petersson trace formula:
\begin{lem}\cite[Iwaniec, Luo, Sarnak, Lemma 10.1]{iwaniecLowLyingZeros2000}
For any positive numbers $m,n$ we have
\begin{align}\label{Petersson}\sum_{k \equiv 0 \mod{2}} &2h\Big(\frac{k-1}{K}\Big) \frac{2\pi^2}{k-1}\sum_{f\in H_k}\frac{ \lambda_f(m)\lambda_f(n)}{L(1, \sym^2f)}=\\
=&\hat{h}(0)K1_{m=n} -\pi^{1/2}(mn)^{-1/4}K \Im \Big(e^{-2\pi i/8}\sum_{c=1}^\infty \frac{S(m,n;c)}{\sqrt{c}} e_c(2\sqrt{mn})\hbar\Big(\frac{cK^2}{8\pi\sqrt{mn}}\Big)\Big)+ \nonumber \\
&+ O\Big(\frac{\sqrt{mn}}{K^4} \cdot \int_{-\infty}^\infty v^4 |\hat{h}(v)|dv+1_{m=n}\Big), \nonumber
\end{align}
where $\hat{h}$ denotes the Fourier transform of $h$ and $\hbar(v)=\int_0^\infty \frac{h(\sqrt{u})}{\sqrt{2\pi u}} e^{iuv} du$.
\end{lem}
\begin{remark}We kept the dependency on $h$ explicit in the error term, as our weight function will depend on $n,m$ and $K$. Similar computations are also done by Khan in \cite{khanNonvanishingSymmetricSquare2010} (see for example Lemma 2.6 and expression (2.29) therein).
\end{remark}
\begin{remark}\label{decay}
Integrating by parts several times shows that $\hbar(v) \ll_{A}v^{-A}$ for any $A>0$. In particular, the second term on the right hand side of \eqref{Petersson} is absorbed in the error term if $cK^2/\sqrt{mn}>K^\varepsilon$. In our case $mn$ will be of size $K^4$ and so this effectively restricts the range of $c$ to $c \ll K^\varepsilon$.
\end{remark}

\subsection{Variance Computation}\label{VarianceComputation}
Now we compute the main term of the variance $V(\psi_1, \psi_2)$, which is given by the averaged shifted convolutioin problem
\begin{equation}\label{Variance}
\mathcal{M}(\psi_1, \psi_2)=\sum_{k \equiv 0 \mod{2}}h\Big(\frac{k-1}{K}\Big) \sum_{f \in H_k} L(1, \sym^2f) S_{\psi_1}S_{\psi_2}
\end{equation}
with
 \begin{align*}S_{\psi}=&\frac{\pi}{2L(1,\sym^2f)}\sum_{\ell \neq 0}\sum_{n}\frac{\lambda_f(n)\lambda_f(n+\ell)}{\sqrt{n(n+\ell)}}\exp\Big(\frac{-k\ell^2}{2(2n+\ell)}\Big)\psi\Big(\frac{k}{2\pi (2n+\ell)}\Big)+O_\psi(k^{-1/2+\epsilon})\\
= &\frac{\pi}{2L(1,\sym^2f)}\sum_{\substack{\ell, d \\ d|\ell}} \sum_{n} \frac{\lambda_f(n(n+\ell/d))}{\sqrt{d^2 n(n+\ell/d)}} \exp\Big( \frac{-k \ell^2}{2(2nd+ \ell)^2}\Big)\psi \Big( \frac{k}{2\pi (2nd+\ell)}\Big) +O_\psi(k^{-1/2+\epsilon})\\
 = &\frac{\pi}{2L(1,\sym^2f)} \sum_{d} \sum_{m} \sum_{n} \frac{\lambda_f(n(n+m))}{d\big(n(n+m)\big)^{1/2}}\exp\Big(-\frac{km^2}{2(2n+m)^2}\Big)\psi\Big(\frac{k}{2\pi d (n+m)}\Big)+O_\psi(k^{-1/2+\varepsilon})
 \end{align*}
For the second equality we used the Hecke relations 
 $$\lambda_f(n)\lambda_f(n+\ell)=\sum_{d|n, d|\ell} \lambda_f\Big(\frac{n(n+\ell)}{d^2}\Big)$$ and replaced $n$ by $nd$. For the third equality we wrote $\ell_i/d_i=m_i$.
After expanding $S_{\psi_1}, S_{\psi_2}$ we see that the main term of  $\mathcal{M}(\psi_1, \psi_2)$ is equal to
$$ \sum_{\substack{n_1,n_2\\ d_1,d_2 \\m_1,m_2}} \sum_{k\equiv 0\mod{2}} h^*\Big(\frac{k-1}{K}\Big)  \frac{2\pi^2}{k-1} \sum_{f\in H_k}\frac{1}{L(1, \sym^2f)} \frac{\lambda_f(n_1(n_1+m_1))\lambda_f(n_2(n_2+m_2))}{d_1d_2\big(n_1(n_1+m_1)n_2(n_2+m_2)\big)^{1/2}}$$
with 
\begin{align*}
h^*(t)&=h^*_{n_1,n_2,d_1,d_2,m_1,m_2,K}(t)\\
&=h(t)\frac{tK}{8}\psi_1\Big(\frac{tK}{2\pi d_1(n_1+m_1)}\Big)\psi_2\Big(\frac{tK}{2\pi d_2(n_2+m_2)}\Big)\exp\Big(-\frac{tKm_1^2}{2(2n_1+m_1)^2}-\frac{tKm_2^2}{2(2n_2+m_2)^2}\Big).
\end{align*}
We can now apply the averaged Petersson trace formula (Lemma \ref{Petersson}) so that $\mathcal{M}(\psi_1, \psi_2)= \mathcal{D} + \mathcal{OD}$ with the diagonal
\begin{align}\label{diagonalterm}\mathcal{D}=K\sum_{\substack{n_1,n_2\\ d_1,d_2 \\m_1,m_2}}\frac{1_{n_1(n_1+m_1)=n_2(n_2+m_2)}}{d_1d_2\big(n_1(n_1+m_1)n_2(n_2+m_2)\big)^{1/2}} \cdot \widehat{h^*}(0)
\end{align}
and the off-diagonal
\begin{align}
\mathcal{OD}=- \sqrt{\pi} K \Im\bigg(&e^{-2\pi i/8}\sum_{\substack{n_1,n_2\\ d_1,d_2 \\m_1,m_2}} \sum_{c=1}^\infty \frac{S(n_1(n_1+m_1),n_2(n_2+m_2);c)}{\sqrt{c}} e_c(2\sqrt{n_1(n_1+m_1)n_2(n_2+m_2)})\cdot \\
\times&\frac{1}{d_1d_2 \big(n_1(n_1+m_1)n_2(n_2+m_2)\big)^{3/4}}\cdot  \hbar^*\Big(\frac{c K^2}{8\pi \sqrt{n_1(n_1+m_1)n_2(n_2+m_2)}}\Big)\bigg),\nonumber
\end{align}
where $\hbar^*(v)=\int_0^\infty \frac{h^*(\sqrt{u})}{\sqrt{2\pi u}}e^{iuv} du$.
Now that we have established a formula 
\begin{equation}\label{main}
\mathcal{M}(\psi_1,\psi_2)=\mathcal{D}+\mathcal{OD},
\end{equation} we start by evaluating the diagonal term.
 \subsection{Evaluating the Diagonal}
 \begin{lem}\label{diagonal} If $\mathcal{D}$ is given by \eqref{diagonalterm} then
\begin{align*} 
\mathcal{D}=&K^{3/2}\log K\cdot \frac{\sqrt{2}\pi}{32}\tilde{\psi_1}(0)\tilde{\psi_2}(0)  \cdot \int_0^\infty \frac{h(\sqrt{u})u^{1/4}}{\sqrt{2\pi u}} du +\\
&+K^{3/2} \frac{\sqrt{2}\pi}{64}\tilde{\psi_1}(0)\tilde{\psi_2}(0)\int_0^\infty \frac{h(\sqrt{u})u^{1/4}}{\sqrt{2\pi u}} \log(u) du+\\
&+K^{3/2}\int_0^\infty \frac{h(\sqrt{u})u^{1/4}}{\sqrt{2\pi u}} du \cdot \Big(\frac{\sqrt{2}\pi}{16}\Big(\frac{3}{2} \gamma - \log (4\pi)\Big) \tilde{\psi_1}(0)\tilde{\psi_2}(0) +\frac{\sqrt{2}\pi}{16} \tilde{\psi_1}(0)\tilde{\psi_2}'(0) \Big)+\\
&+K^{3/2} \int_0^\infty \frac{h(\sqrt{u})u^{1/4}}{\sqrt{2\pi u}} du \cdot \frac{\sqrt{2}\pi}{16} \frac{1}{2\pi i} \int_{(1)} \tilde{\psi}_1(-s_2)\tilde{\psi_2}(s_2)\zeta(1-s_2)\zeta(1+s_2)ds_2 +\\
&+O_{\psi_1, \psi_2}(K^{1+\varepsilon})
\end{align*}
as $K \to \infty$.
 \end{lem}
 \begin{proof}
 To evaluate $\mathcal{D}$ (see \eqref{diagonalterm}), we first show that most solutions to $n_1(n_1+m_1)=n_2(n_2+m_2)$ arise from the diagonal, i.e.  $n_1=n_2$ and $m_1=m_2$. We assume that $n_1\neq n_2$ and $m_1\neq m_2$, as otherwise we are done. By completing the square the condition $n_1(n_1+m_1)=n_2(n_2+m_2)$ is equivalent to 
$$(2n_1+m_1)^2-(2n_2+m_2)^2=m_1^2-m_2^2.$$
Since $m_i \ll K^{1/2+\varepsilon}$ there are at most $K^{1+2\varepsilon}$ choices for the integer $M=m_1^2-m_2^2$. Once $M, m_1, m_2$ are fixed, $n_1$ and $n_2$ are determined up to $K^{\varepsilon'}$. To see this abbreviate $A=(n_1+m_1+n_2+m_2)$ and $B=(n_1+m_1-n_2-m_2)$. Then $M=4AB$ and there are at most $K^{\varepsilon'}$ choices for $A$ and $B$ (as there are at most $K^{\varepsilon'}$ divisors of $M$). Now $A, B, m_1, m_2$ are determined and so are $n_1, n_2$.  
It follows that there are at most $K^{1+2\varepsilon+\varepsilon'}$ off-diagonal terms, whose contribution to $\mathcal{D}$ is bounded by
\begin{align*}K \sum_{\substack{n_1, n_2, m_1, m_2\\ n_1\neq n_2, m_1\neq m_2}} \sum_{d_1, d_2} \frac{1_{n_1(n_1+m_1)=n_2(n_2+m_2)}}{d_1d_2\big(n_1(n_1+m_1)n_2(n_2+m_2)\big)^{1/2}} \cdot \widehat{h^*}(0) \ll_{\psi_1, \psi_2} K^{1+2\varepsilon +\varepsilon'}.
\end{align*}
Hence,
\begin{align*}
\mathcal{D}=&\frac{\sqrt{2\pi}K^2}{16} \sum_{n_1, d_1, d_2, m_1}\frac{1}{d_1d_2 n_1(n_1+m_1)}\cdot \\
&\times \int_0^\infty \frac{h(\sqrt{u})\sqrt{u}}{\sqrt{2\pi u}}\psi_1\Big(\frac{\sqrt{u}K}{2\pi d_1(n_1+m_1)}\Big)\psi_2\Big(\frac{\sqrt{u}K}{2\pi d_2(n_2+m_2)}\Big) \exp\Big(- \frac{\sqrt{u}K m_1^2}{(2n_1
+m_1)^2}\Big) du\\
&+O_{\psi_1, \psi_2}(K^{1+\varepsilon}).
\end{align*}

Since, $m_i d_i \ll K^{1/2+\varepsilon}$ and $m_i/n_i \ll K^{1/2+\varepsilon}$ for $i=1,2$ we can simplify the expression for the off-diagonal $\mathcal{D}$ by a Taylor expansion and get
\begin{align*}\mathcal{D}=&\frac{\sqrt{2\pi}K^2}{16} \sum_{n_1, d_1,d_2, m_1} \frac{1}{d_1d_2n_1^2} \int_0^\infty \frac{ h(\sqrt{u})\sqrt{u}}{\sqrt{2\pi u}} \psi_1\Big(\frac{\sqrt{u}K}{4\pi n_1d_1}\Big)\psi_2\Big(\frac{\sqrt{u}K}{4\pi n_1 d_2}\Big)\exp\Big(-\frac{\sqrt{u}Km_1^2}{4n_1^2}\Big)du\\
&+O_{\psi_1, \psi_2}(K^{1+\varepsilon}).
\end{align*}
To evaluate the main term of $\mathcal{D}$ asymptotically we perform an inverse Mellin transform on the smooth compactly-supported functions $\psi_1, \psi_2$ and the exponential function. We then shift the contours and collect the residues. The main term of $\mathcal{D}$ is equal to
\begin{align*}
\frac{\sqrt{2\pi} K^2}{16} \frac{1}{(2\pi i)^3}& \int_{(1/2+\varepsilon)}\int_{(1)}\int_{(1)} \int_0^\infty \frac{h(\sqrt{u})\sqrt{u}}{\sqrt{2\pi u}}\tilde{\psi}_1(s_1)\tilde{\psi_2}(s_2)\Gamma(s_3)\cdot \\
&\cdot \sum_{n_1, d_1,d_2,m_1} \frac{1}{d_1d_2n_1^2}\Big(\frac{\sqrt{u}K}{4\pi n_1 d_1}\Big)^{s_1}\Big(\frac{\sqrt{u}K}{4\pi n_1d_2}\Big)^{s_2}\Big(\frac{4n_1^2}{\sqrt{u}Km_1^2}\Big)^{s_3} du ds_1ds_2ds_3.
\end{align*}
This is turn can be rewritten as
\begin{align*}\frac{\sqrt{2\pi} K^2}{16} \frac{1}{(2\pi i)^3}& \int_{(1/2+\varepsilon)}\int_{(1)}\int_{(1)} \int_0^\infty \frac{h(\sqrt{u})\sqrt{u}}{\sqrt{2\pi u}}u^{(s_1+s_2-s_3)/2} K^{s_1+s_2-s_3} (4\pi)^{-s_1-s_2}4^{s_3}\cdot \\
&\cdot \tilde{\psi}_1(s_1)\tilde{\psi_2}(s_2)\Gamma(s_3)\zeta(1+s_1)\zeta(1+s_2)\zeta(2+s_1+s_2-2s_3)\zeta(2s_3) duds_1ds_2ds_3
\end{align*}
We start by shifting the contour from $\Re(s_3)=1/2+\varepsilon$ to $\Re(s_3)=100$. The integral on the new line $Re(s_3)=100$ is negligible by the rapid decay of $\tilde{\psi_1}(s_1), \tilde{\psi_2}(s_2)$ and the Gamma function (it contributes at most $O_{\psi_1, \psi_2}(K^{-96})$). The simpe pole at $s_3=1/2+s_1/2+s_2/2$ yields the residue
\begin{align}\label{diagonalstep4}\frac{\sqrt{2\pi} K^2}{16} \frac{1}{(2\pi i)^2}&\int_{(1)}\int_{(1)} \int_0^\infty \frac{h(\sqrt{u})\sqrt{u}}{\sqrt{2\pi u}}u^{(-1/2+s_1/2+s_2/2)/2} K^{-1/2+s_1/2+s_2/2} (4\pi)^{-s_1-s_2}4^{1/2+s_1/2+s_2/2} \\
&\cdot \tilde{\psi}_1(s_1)\tilde{\psi_2}(s_2)\Gamma(1/2+s_1/2+s_2/2)\zeta(1+s_1)\zeta(1+s_2)\frac{1}{2}\zeta(1+s_1+s_2)duds_1ds_2 \nonumber
\end{align}
Next we move the line $\Re(s_1)=1$ to $\Re(s)=-2+\varepsilon$ (stopping just before the pole of the Gamma function), picking up simple poles at $s_1=0$ and $s_1=-s_2$ . We use again the rapid decay of $\tilde{\psi_1}(s_1), \tilde{\psi_2}(s_2)$ to show that the new line integral is bounded by $O_{\psi_1, \psi_2}(K^{1+\varepsilon})$.  At $s_1=0$ the residue is
\begin{align}\label{maintermdiagonal}\frac{\sqrt{2\pi} K^2}{16} \frac{1}{2\pi i} \int_{(1)}& \int_0^\infty \frac{h(\sqrt{u})\sqrt{u}}{\sqrt{2\pi u}}u^{(-1/2+s_2/2)/2} K^{-1/2+s_2/2}  (4\pi)^{-s_2}4^{1/2+s_2/2} \\
&\cdot \tilde{\psi}_1(0)\tilde{\psi_2}(s_2)\Gamma(1/2+s_2/2)\zeta(1+s_2)^2\frac{1}{2}duds_2 \nonumber
\end{align}
We follow up with the shift from $\Re(s_2)=1$ to $\Re(s_2)=-1+\varepsilon$ and pick up a double pole at $s_2=0$. The error term from the line at $\Re(s_2)=-1+\varepsilon$ is again $O_{\psi_1, \psi_2}(K^{1+\varepsilon})$. To compute the residue at the double pole we use the expansion 
$\zeta(1+s_2)^2=\frac{1}{s_2^2}+\frac{2\gamma}{s}+\cdots$ for $s_2$ close to $0$, where $\gamma$ is the Euler-Mascheroni constant. The residue is then given by
$$\frac{\sqrt{2\pi} K^{3/2}}{16} \cdot  \int_0^\infty \frac{h(\sqrt{u})u^{1/4}}{\sqrt{2\pi u}} \cdot \lim_{s_2 \to 0} \frac{d}{ds} \Big(s^2 \cdot \Big(\frac{K\sqrt{u}}{4\pi^2}\Big)^{\frac{s_2}{2}} \tilde{\psi_2}(s_2)\Gamma(1/2+s_1/2) \cdot \Big(\frac{1}{s_2^2}+\frac{2\gamma}{s}+\cdots \Big)\Big)du$$
The limit is equal to
\begin{align*}\lim_{s_2\to 0} \Big(\frac{K\sqrt{u}}{4\pi^2}\Big)^{\frac{s_2}{2}}\cdot \Big( &\frac{1}{2} \log \Big(\frac{K\sqrt{u}}{4\pi^2}\Big)\tilde{\psi_2}(s_2)\Gamma(1/2+s_2/2) + \\
&+ \tilde{\psi_2}'(s_2)\Gamma(1/2+s_2/2) +\frac{1}{2}\tilde{\psi_2}(s_2)\Gamma'(1/2+s_2/2)+2\gamma \tilde{\psi_2}(s_2)\Gamma(1/2+s_2/2)\Big)
\end{align*}
We evaluate the limit, using $\Gamma'(1/2)=\sqrt{\pi} (-\gamma-\log4)$, to
$$\frac{\sqrt{\pi}}{2}\tilde{\psi_2}(0)  \log K  +  \frac{\sqrt{\pi}}{4}\tilde{\psi_2}(0) \log u +\sqrt{\pi}\tilde{\psi_2}(0) \Big(\frac{3}{2} \gamma - \log (4\pi)\Big) + \sqrt{\pi}\tilde{\psi_2}'(0)$$
Thus \eqref{maintermdiagonal} is equal to
\begin{align*}
K^{3/2}&\log K\cdot \frac{\sqrt{2}\pi}{32}\tilde{\psi_1}(0)\tilde{\psi_2}(0)  \cdot \int_0^\infty \frac{h(\sqrt{u})u^{1/4}}{\sqrt{2\pi u}} du \\
+&K^{3/2} \frac{\sqrt{2}\pi}{64}\tilde{\psi_1}(0)\tilde{\psi_2}(0)\int_0^\infty \frac{h(\sqrt{u})u^{1/4}}{\sqrt{2\pi u}} \log(u) du+\\
+&K^{3/2}\int_0^\infty \frac{h(\sqrt{u})u^{1/4}}{\sqrt{2\pi u}} du \cdot \Big(\frac{\sqrt{2}\pi}{16}\Big(\frac{3}{2} \gamma - \log (4\pi)\Big) \tilde{\psi_1}(0)\tilde{\psi_2}(0) +\frac{\sqrt{2}\pi}{16} \tilde{\psi_1}(0)\tilde{\psi_2}'(0) \Big)\\
+&O_{\psi_1, \psi_2}(K^{1+\varepsilon})
\end{align*}
This is our main term. There is another term of size $K^{3/2}$ coming from the residue of expression \eqref{diagonalstep4} at $s_1=-s_2$. This residue is given by
\begin{align}\label{diagonalstep6}K^{3/2} \int_0^\infty \frac{h(\sqrt{u})u^{1/4}}{\sqrt{2\pi u}} du \cdot \frac{\sqrt{2}\pi}{16} \frac{1}{2\pi i} \int_{(1)} \tilde{\psi}_1(-s_2)\tilde{\psi_2}(s_2)\zeta(1-s_2)\zeta(1+s_2)ds_2 
\end{align}
This completes the proof of the lemma. 
\end{proof}
\subsection{Auxiliary Lemmas}
In the following section we record some lemmas that we use to compute the off-diagonal term asymptotically. We start with some observations regarding the function $\hbar(v)$, appearing in the off-diagonal term. For any complex number $w$ define the function
$$\hbar^{\Re}_{w}(v)=\int_0^\infty \frac{h(\sqrt{u})}{\sqrt{2\pi u}} u^{w/2}\cos(uv)du.$$
For $w=0$ this is the real part of $\hbar(v)$. The Mellin transform of this function and its properties were evaluated by Khan \cite[Lemma 3.5]{khanNonvanishingSymmetricSquare2010} (and also Das-Khan \cite[sec. 2.6]{dasThirdMomentSymmetric2018}). Note that Khan and Das-Khan treat $\hbar_w(v)$ but the observations also go through  for the real part that we consider. As in \cite[sec. 2.6]{dasThirdMomentSymmetric2018} we have by repeated integration by parts the bound
\begin{equation}\label{hbarreal}
\frac{\partial^j}{\partial v^j} \hbar^{\Re}_{w}(v) \ll (1+|w|)^A|v|^{-A}
\end{equation} for any non-negative integer $j,A$ and the implied constant depending on $\Re(w), j, A$. 
We denote the Mellin transform of $\hbar^{\Re}_{w}$ by
$$\tilde{\hbar}^{\Re}_w(s)=\int_0^\infty \hbar^{\Re}_w(v)v^s \frac{dv}{v}.$$
The bound \eqref{hbarreal} implies that the Mellin transform is absolutely convergent and holomorphic for $\Re(s)>0$. Integrating by parts several times and using again the bound \eqref{hbarreal} shows that the Mellin transform deacys rapidly. More precisiely we have
$$\tilde{\hbar}^{\Re}_w(s) \ll (1+|w|)^{A+\Re(s)+1}(1+|s|)^{-A},$$
with the implied constants depending on $\Re(w)$ and $A$. 
By Mellin inversion we have for $c>0$
\begin{equation}\label{inverseMellin}\hbar^{\Re}_w(v)= \frac{1}{2\pi i}\int_{(c)}\tilde{\hbar}^{\Re}_w(s)v^{-s} ds.
\end{equation}
As in \cite[Lemma 3.5]{khanNonvanishingSymmetricSquare2010} we can explicitly evaluate the Mellin transform of $\hbar^\Re_w$ within the range $0 < \Re(s) <1$. There we get
$$\tilde{\hbar}^{\Re}_w(s)=\int_0^\infty \frac{h(\sqrt{u})}{\sqrt{2\pi u}}u^{w/2}\Gamma(s)\cos(\pi s/2) du.$$

The next two lemmas will be useful to treat the exponential sum in the off-diagonal term.
\begin{lem}[Poisson summation]\label{Poisson}Let $f$ be a rapidly decaying, smooth function, then
$$\sum_{n\equiv a\mod{c}} f(n)= \frac{1}{c} \sum_{n} \hat{f}\Big(\frac{n}{c}\Big)e_c(an),$$
where $\hat{f}(\xi)=\int_{-\infty}^\infty f(x)e(-x\xi) dx$ denotes the Fourier transform of $f$.
\end{lem}
\begin{proof}
This follows immediately from the classical Poisson summation formula and noting that $n\equiv a\mod{c}$ is a shifted lattice of $\mathbb{Z}$. 
\end{proof}
We detect cancellation in the off-diagonal term with the stationary phase method. We use the following version of Blomer, Khan and Young, which is a special case of their Proposition 8.2 in \cite{blomerDistributionMassHolomorphic2013b}
\begin{lem}[Stationary phase] \label{stationaryphase}
Let $X,Y,V,V_1,Q>0$ and $Z:=Q+X+Y+V_1+1$, and assume that 
$$Y \geq Z^{3/20},\quad V_1 \geq V\geq \frac{QZ^{1/40}}{Y^{1/2}}.$$
Suppose that $h$ is a smooth function on $\mathbb{R}$ with support on an interval $J$ of length $V_1$, satisfying
$$h^{(j)}(t)\ll_j XV^{-j}$$ for all $j \in \mathbb{N}_0$. Suppose $f$ is a smooth function on $J$ such that there exists a unique point $t_0\in J$ such that $f'(t_0)=0$, and furthermore
$$f''(t) \gg YQ^{-2}, \quad f^{(j)}(t)\ll_j YQ^{-j}, \quad \text{ for } j\geq 1 \text{ and } t \in J.$$
Then 
$$\int_{-\infty}^\infty h(t)e^{2\pi if(t)}dt = e^{\sgn(f''(t_0))\cdot \pi i/4}\frac{e^{2\pi if(t_0)}}{\sqrt{|f''(t_0)|}} h(t_0)+ O\Big(\frac{Q^{3/2}X}{Y^{3/2}}\cdot\big(V^{-2}+(Y^{2/3}/Q^2)\big)\Big).$$
In particular, we also have the trivial bound
$$\int_{-\infty}^\infty h(t) e^{2\pi i f(t)} dt \ll \frac{XQ}{\sqrt{Y}}+1.$$
\end{lem}
\begin{proof}
See Proposition 8.2 in \cite{blomerDistributionMassHolomorphic2013b}, with $\delta=1/20$ and $A$ sufficiently large. We bounded the contribution of the non-leading terms in the asymptotic expansion of \cite[Eq. 8.9]{blomerDistributionMassHolomorphic2013b}) trivially by $O\Big(\frac{Q^{3/2}X}{Y^{3/2}}\cdot\big(V^{-2}+(Y^{2/3}/Q^2)\big)\Big)$.
\end{proof}

The last lemma we need concerns the Kloosterman sum over arithmetic progressions. We have 
\begin{lem}\label{Kloostermansum}Let $S(a,b; c)$ denote the classical Kloosterman sum, then
$$\sum_{\substack{a_1 \mod{c},\\ a_2 \mod{c}}} \sum_{\substack{b_1 \mod{c} ,\\ b_2 \mod{c}}} S(a_1(a_1+b_1),a_2(a_2+b_2);c)e_c(2a_1a_2+a_1b_2+a_2b_1)=c^3\phi(c),$$
where $\phi(c)$ is Euler's totient function.
\end{lem}
\begin{proof}
Let us call $\mathcal{T}$ the sum we must calculate. First we open the Kloosterman sum and get 
\begin{align*}
\mathcal{T}=&\sum_{\substack{a_1 \mod{c},\\ a_2 \mod{c}}} \sum_{\substack{b_1 \mod{c} ,\\ b_2 \mod{c}}} \sum_{\substack{x \mod{c} \\ (x,c)=1}} e_c\big(a_1(a_1+b_1)\bar{x}+a_2(a_2+b_2)x+2a_1a_2+a_1b_2+a_2b_1\big).
\end{align*}
Here $x\bar{x}\equiv 1 \mod c$. Since $(x,c)=1$ we can substitute $a_1$ with $a_1x$ and $b_1$ with $b_1x$. We get
\begin{align*}
\mathcal{T}=&\sum_{\substack{a_1 \mod{c},\\ a_2 \mod{c}}} \sum_{\substack{b_1 \mod{c} ,\\ b_2 \mod{c}}} \sum_{\substack{x \mod{c} \\ (x,c)=1}} e_c\big(a_1^2x+a_1b_1x+a_2^2x+a_2b_2x+2a_1a_2x+a_1b_2x+a_2b_1x\big)\\
=&\sum_{\substack{a_1 \mod{c},\\ a_2 \mod{c}}} \sum_{\substack{b_1 \mod{c} ,\\ b_2 \mod{c}}} \sum_{\substack{x \mod{c} \\ (x,c)=1}}e_c\big((a_1+a_2)^2x+(a_1+a_2)(b_1+b_2)x\big)\\
=&~c^3\phi(c)
\end{align*}
To obtain the last equality we used orthogonality when summing over $b_1$ and $b_2$, i.e.
$$\sum_{b\mod{c}} e_c\big(b(a_1+a_2)x\big) =
\begin{cases}
c &\text{ if } (a_1+a_2)x \equiv 0 \mod{c}\\
0 &\text{ otherwise}
\end{cases}
$$
and so $a_1\equiv-a_2 \mod{c}$. 
\end{proof}
\begin{remark}
We highlight here that the additional summation over $b_1, b_2$ (that originally stems from averaging over the shifts) is crucial here. In contrast to that, Luo and Sarnak (\cite[Appendix A.2]{luoQuantumVarianceHecke2004a}) work  for the full fundamental domain with fixed shifts. The corresponding summation over the Kloosterman sum then reduces ``only'' to an expression involving Sali\'e sums. Lemma \ref{Kloostermansum} might be seen as the reason why we can obtain a comparably clean formula for the quantum variance of the vertical geodesic.
\end{remark}
\subsection{Evaluating the Off-Diagonal}
Our goal is to obtain an asymptotic formula for the off-diagonal $\mathcal{OD}$ given by

 \begin{align}\label{offdiagonal2}
-\sqrt{\pi} \Im \bigg(&e^{-2\pi i/8} K \sum_{d_1,d_2}\sum_{n_1,n_2} \sum_{m_1, m_2} \sum_{c} \frac{S(n_1(n_1+m_1),n_2(n_2+m_2);c)}{\sqrt{c}}e_c(2n_1n_2+n_1m_2+n_2m_1)\cdot  \\
& \times \frac{e_c\big(f(n_1,n_2,m_1,m_2)\big)}{d_1d_2 \big(n_1(n_1+m_1)n_2(n_2+m_2)\big)^{3/4}} \cdot \hbar_{d_1, d_2}^*\Big(\frac{c K^2}{8\pi \sqrt{n_1(n_1+m_1)n_2(n_2+m_2)}}\Big)\bigg) \nonumber
 \end{align}
 where $$f(n_1, n_2, m_1, m_2)=2\sqrt{n_1(n_1+m_1)n_2(n_2+m_2)}-2n_1n_2-n_1m_2-n_2m_1$$
 and 
 \begin{align}\label{weightfunction}  \hbar_{d_1, d_2}^*\Big(&\frac{c K^2}{8\pi \sqrt{n_1(n_1+m_1)n_2(n_2+m_2)}}\Big)=
 \\
 \frac{K}{8}&\int_0^\infty \frac{h(\sqrt{u})\sqrt{u}}{\sqrt{2\pi u}} \psi_1\Big(\frac{\sqrt{u}K}{2\pi d_1(n_1+m_1)}\Big)\psi_2\Big(\frac{\sqrt{u}K}{2\pi d_2(n_2+m_2)}\Big) \cdot \nonumber\\
 &\times  \exp\Big(- \frac{\sqrt{u}K m_1^2}{(2n_1+m_1)^2}\Big) \exp\Big(- \frac{\sqrt{u}K m_2^2}{(2n_2+m_2)^2}\Big) e^{iu \frac{cK^2}{8\pi (n_1(n_1+m_1)n_2(n_2+m_2))^{1/2}}} du. \nonumber
 \end{align}
This task requires several intermediate steps.

First we will detect square-root cancellation in the exponential sum 
$$\sum_{n_1, n_2 \asymp K} e_c\big(f(n_1, n_2, m_1, m_2)\big)$$
when $m_1, m_2$ are large. To see that this is possible, it is helpful to keep the critical ranges $m_i\asymp \sqrt{K}$ and $n_i\asymp K$ for $i=1,2$ in mind. We then have the following heuristic that guides our further analysis:
$$f(n_1, n_2, m_1, m_2)=-\frac{m_1^2n_2}{2n_1}-\frac{m_2^2n_1}{2n_2}+ \ldots$$
$$\frac{\partial}{\partial n_1} f(n_1, n_2, m_1, m_2) \approx \frac{m_1^2n_2}{n_1^2}-\frac{m_2^2}{2n_2} \asymp O(1)\quad \text{and}\quad \frac{\partial^2}{\partial n_1^2}f(n_1, n_2, m_1, m_2) \approx -\frac{3m_1^2n_2}{n_1^3} \asymp \frac{1}{K}.$$
From these bounds on the derivatives we can see that we expect square root cancellation in the summation over $n_1$ (see \cite[Corollary 8.12]{IwaniecKowalskiAnalytica}). To make this heuristic precise and to compute an asymptotic formula for the off-diagonal $\mathcal{OD}$, we will use the Poisson summation formula and the stationary phase method, which is the subject of the following lemmas.

\begin{lem}\label{afterPoisson} Let $\mathcal{OD}$ be defined as in \eqref{offdiagonal2}. Then 
\begin{align*}\mathcal{OD}=
-\sqrt{\pi}K  \Im \bigg(&e^{-2\pi i/8} \sum_{c} \sum_{\substack{b_1 \mod{c}, \\ b_2 \mod{c}}}\sum_{\substack{a_1 \mod{c}, \\ a_2 \mod{c}}} \frac{S(a_1(a_1+b_1),a_2(a_2+b_2);c)}{c^{3/2}}e_c(2a_1a_2+a_1b_2+a_2b_1) \cdot \\
& \times \sum_{d_1,d_2}\frac{1}{d_1d_2}  \sum_{\substack{m_1 \equiv b_1 \mod{c}\\m_2 \equiv b_2 \mod{c}}}\sum_{n_2 \equiv a_2 \mod{c}} \sum_{v\in \mathbb{Z}} \mathcal{I}_v(n_2, m_1, m_2, d_1, d_2, c) \bigg)
\end{align*}
with 
\begin{align}\label{integral}
\mathcal{I}_v(n_2, m_1, m_2, d_1, d_2, c):=\int_{-\infty}^\infty &\frac{e_c\big(f(x, n_2,m_1,m_2)-v(x+a_1)\big)}{d_1d_2 \big(x(x+m_1)n_2(n_2+m_2)\big)^{3/4}}\hbar_{d_1, d_2}^*\Big(\frac{c K^2}{8\pi \sqrt{x(x+m_1)n_2(n_2+m_2)}}\Big)dx 
\end{align}
\end{lem}
\begin{proof}
We work with expression \eqref{offdiagonal2} and split the variables $n_1, n_2, m_1, m_2$ into residue classes modulo $c$. The off-diagonal $OD$ is then given by
  \begin{align}\label{Offdiagonalexpression}
 -\sqrt{\pi}K  \Im \bigg(&e^{-2\pi i/8} \sum_{c} \sum_{\substack{b_1 \mod{c}, \\ b_2 \mod{c}}}\sum_{\substack{a_1 \mod{c}, \\ a_2 \mod{c}}} \frac{S(a_1(a_1+b_1),a_2(a_2+b_2);c)}{\sqrt{c}}e_c(2a_1a_2+a_1b_2+a_2b_1) \cdot \\
& \times \sum_{d_1,d_2} \sum_{\substack{m_1 \equiv b_1 \mod{c}\\m_2 \equiv b_2 \mod{c}}}\sum_{\substack{n_1 \equiv a_1 \mod{c}\\n_2 \equiv a_2 \mod{c}}}\frac{e_c\big(f(n_1,n_2,m_1,m_2)\big)}{d_1d_2 \big(n_1(n_1+m_1)n_2(n_2+m_2)\big)^{3/4}} \cdot\\&\times \hbar_{d_1, d_2}^*\Big(\frac{c K^2}{8\pi \sqrt{n_1(n_1+m_1)n_2(n_2+m_2)}}\Big)\bigg). \label{exposum}
 \end{align}
Here we used the fact that the Kloosterman sum only depends on the residue classes modulo $c$. The lemma follows now upon applying the Poisson summation formula (see Lemma \ref{Poisson}) to the summation over $n_1$.
\end{proof}

Next we will analyze the quantity
\begin{equation}\label{maintermanderrorterm}
\sum_{n_2 \equiv a_2\mod{c}} \sum_{v\in \mathbb{Z}} \mathcal{I}_v(n_2, m_1, m_2, d_1, d_2, c)
\end{equation} with the stationary phase method.

We record here some restrictions on the variables that will be useful for the further analysis. Since $h, \psi_1, \psi_2$ are compactly supported on $\mathbb{R}^+$ we have that $n_id_i \asymp K$ for $i=1/2$. Additionally, $d_i \leq K^\delta$ for some $\delta>0$, as noted in Remark \ref{decay}. Indeed, if $d_i>K^\delta$ then $n_i\leq K^{1-\delta}$ and $cK^2/\sqrt{n_1(n_1+m_1)n_2(n_2+m_2)} \geq K^{2\delta}$ and the off-diagonal term can be absorbed in the error term. We will choose $\delta = 1/32$ for convenience. From the observations in Remark \ref{decay} we also have the condition $\frac{cK^2}{n_1n_2} \ll K^\epsilon$. The exponential functions ensure that $m_i \leq n_i/K^{1/2-\varepsilon}\leq K^{1/2+\varepsilon}$, as otherwise we have exponential decay.  For technical purposes we also impose a lower bound on $m_i$. If both variables $m_1, m_2$ are small than the exponential  $e_c(\sqrt{n_1(n_1+m_1)n_2(n_2+m_2)})$ is essentially smooth and the analysis is similar as in the case of full fundamental domain (see \cite{luoQuantumVarianceHecke2004a}). If  $m_i \leq K^{1/8}$ for $i=1,2$, then we can bound $OD$ trivially by $K^{5/4+\varepsilon}$, which is smaller than the expected main term of size $K^{3/2}$. In our stationary phase analysis we will only need one of the variables $m_i$ to be large. Since $OD$ is symmetric in $m_i/n_i$ we can assume without loss of generality that $m_1\geq K^{1/8}$. The most important case is of course when both $m_1$ and $m_2$ are of size $\sqrt{K}$ and so one should view this restriction only as a technical convenience. To summarize, we will work subsequently under the following conditions:
\begin{equation}\label{cond1}n_id_i \asymp K \quad \text{and}\quad d_i \leq K^{1/32}\quad \text{for } i=1,2,\end{equation}
\begin{equation}\label{cond2}K^{1/8} \leq m_1 \leq K^{1/2+\varepsilon} \quad \text{and}\quad 1\leq m_2 \leq K^{1/2+\varepsilon},\end{equation}
\begin{equation}\label{cond3}\frac{cK^2}{n_1n_2} \ll K^\varepsilon \quad \text{and in particular }\quad c \ll K^\varepsilon.\end{equation}

A stationary phase analysis leads to the following lemma:
\begin{lem}\label{StationaryInt}
Let $\mathcal{I}_v(n_2, m_1, m_2, d_1, d_2, c)$ be defined as in \eqref{integral}. Then
\begin{align*}\mathcal{I}_v(n_2, m_1, m_2, d_1, d_2, c)=&\frac{e_c\Big(a_1v+\frac{m_1}{2}\big(v+m_2-\sqrt{(v+m_2)^2+4vn_2}\big)\Big)e^{-\pi i/4}\sqrt{c}}{\sqrt{|f''(x_v^*(n_2))|}\cdot x_v^{*}(n_2)^{3/2}\cdot  n_2^{3/2}} \hbar_{d_1, d_2}^*\Big(\frac{cK^2}{8\pi  n_2x_v^*(n_2)}\Big)\\&+ O(K^{-2+\varepsilon})
\end{align*}
with $x_v^*(n_2)=\frac{m_1}{2}\Big(-1+\frac{v+m_2+2n_2}{\sqrt{(v+m_2)^2+4vn_2}}\Big)$.
\end{lem}
\begin{proof}

We now need to check the conditions of the stationary phase Lemma \ref{stationaryphase} so that we can apply it to the quantity $\mathcal{I}_v(n_2, m_1, m_2, d_1, d_2, c)$ . The stationary points $x_v^*(n_2)$ for fixed $v$ are the solutions to the equation
$$\frac{\partial}{\partial x}\big(f(x,n_2,m_1,m_2)-v(x-a_1)\big)=0 $$
and are given by
\begin{equation}\label{stationarypoint}
x_v^*(n_2)=\frac{m_1\big(m_2+2n_2+v-\sqrt{(v+m_2)^2+4vn_2}\big)}{2\sqrt{(v+m_2)^2+4vn_2}}=\frac{m_1 n_2}{\sqrt{(v+m_2)^2+4vn_2}}\cdot \big(1+O(K^{-1/2+1/32+\varepsilon})\big).
\end{equation}
From a Taylor expansion of $f(x, n_1, m_1, m_2)$ we see that
$$f(x, n_2, m_1, m_2)=-\frac{1}{4}\frac{n_2 m_1^2}{x} -\frac{1}{4}\frac{xm_2^2}{n_2} +\frac{m_1m_2}{2} -\frac{1}{8}\frac{m_1^2m_2}{x}-\frac{1}{8}\frac{m_2^2m_1}{n_2}+\ldots$$ and thus
$$\frac{\partial}{\partial x} f(x,n_1, m_1, m_2)= \frac{1}{4} \frac{n_2m_1^2}{x^2} -\frac{1}{4} \frac{m_2^2}{n_2} +\frac{1}{8} \frac{m_1^2m_2}{x^2}+\ldots \ll K^\varepsilon.$$
From the bound on the first derivative of $f$ it follows also that $v \ll K^\varepsilon$.
The second derivate of the phase function $f$ is given by
$$\frac{\partial^2}{\partial x^2} f(x,n_2,m_1,m_2)=-\frac{m_1^2}{2}\sqrt{\frac{n_2(n_2+m_2)}{(x(x+m_1))^3}}.$$
We note here that 
\begin{equation}\label{derivativebounds}\Big|\frac{\partial^2}{\partial x^2}f(x,n_2,m_1,m_2)\Big| \asymp \frac{m_1^2n_2}{x^3} \quad \text{ and } \quad \Big|\frac{\partial^j}{\partial x^j}f(x, n_2,m_1,m_2)\Big| \ll_j \frac{m_1^2n_2}{x^{1+j}}
\end{equation}
for any integer $j\geq 2$.
If we evaluate the second derivative at the stationary point $x_v^*$ we get
\begin{equation}\label{secondderivative}
f''(x_v^*)=-\frac{((v+m_2)^2+4vn_2)^{3/2}}{2m_1n_2(m_2+n_2)}=-\frac{((v+m_2)^2+4vn_2)^{3/2}}{2m_1n_2^2}\cdot (1+O(K^{-1/2+\epsilon})).
\end{equation}
After a tedious computation we see that the exponential evaluated at the stationary point equals
$$e_c\big(f(x_v^*, n_2,m_1,m_2)-v(x_v^*-a_1)\big)=e_c\Big(a_1v+\frac{m_1}{2}\big(v+m_2-\sqrt{(v+m_2)^2+4vn_2}\big)\Big).$$
For the stationary phase analysis we also require bounds on the derivatives of  the weight function $\hbar_{d_1,d_2}^*\big(\frac{cK^2}{8\pi (x(x+m_1)n_2(n_2+m_2))^{1/2}}\big)$, defined in \eqref{weightfunction}.
Conditions \eqref{cond1}, \eqref{cond2} and \eqref{cond3} yield
 \begin{align}\label{derivatives}\frac{\partial^j}{\partial x^j}&\frac{1}{ \big(x(x+m_1)n_2(n_2+m_2)\big)^{3/4}}\hbar_{d_1, d_2}^*\Big(\frac{c K^2}{8\pi \sqrt{x(x+m_1)n_2(n_2+m_2)}}\Big) \\
 &\ll \frac{1}{x^{3/2} n_2^{3/2}}K\Big(\frac{1}{x} + \frac{Km_1^2}{x^3}+\frac{cK^2}{n_2x^2}\Big)^j \nonumber\\
& \ll \frac{1}{x^{3/2} n_2^{3/2}} K\Big(\frac{K^\epsilon}{x}\Big)^j \nonumber
 \end{align} 
 for any fixed integer $j\geq 0$. 

We are now in the position to determine the required parameters $X, Y, V, V_1, Q$ and $Z$ of Lemma \ref{stationaryphase}.
From the bounds on the derivatives \eqref{derivativebounds} and the fact that $x\asymp K/d_1$ we see that $Y=\frac{m_1^2n_2d_1}{cK}$ and $Q=\frac{K}{d_1}$. From \eqref{derivatives} we see that $X=d_1^{3/2} n_2^{-3/2} K^{-1/2}$, $V_1=V=\frac{K^{1-\varepsilon}}{d_1}$ and $Z\ll K^{1+\varepsilon+1/32}$. With these choices we also see that the condition $V_1 \geq \frac{QZ^{1/40}}{Y^{1/2}}$ is satisfied, as long as $d_2 \leq K^{1/32}$, say. Indeed, with the lower bound on $m_1\geq K^{1/8}$ we have
$$\frac{QZ^{1/40}}{Y^{1/2}} \ll \frac{K\sqrt{cK}Z^{1/40}}{d_1m_1 \sqrt{n_2}} \ll \frac{K}{d_1} \cdot \frac{K^{\varepsilon}\sqrt{d_2}K^{1/30}}{m_1}\ll \frac{K}{d_1} \cdot K^{1/64+1/30+\varepsilon-1/8} \ll V K^{-1/16}.$$

We can now apply the stationary phase lemma (Lemma \ref{stationaryphase}) and see that
\begin{align*} \mathcal{I}_v(n_2, m_1, m_2, d_1, d_2, c)=&\frac{e_c\big(f(x_v^*(n_2), n_2,m_1,m_2)-v(x_v^*-a_1)\big)e^{-\pi i/4}\sqrt{c}}{\sqrt{|f''(x_v^*(n_2))|} \cdot \big(x_v^{*}(n_2)(x_v^*(n_2)+m_1)n_2(n_2+m_2)\big)^{3/4}}\\
& \times \hbar^*\Big(\frac{cK^2}{8\pi \sqrt{x_v^*(n_2)(x_v^*(n_2)+m_1)n_2(n_2+m_2)}}\Big)+ O(K^{-2}).
\end{align*}
Since $x_v^*(n_2) \asymp n_1$ we can simplify the above result slightly by using a Taylor expansion and conditions \eqref{cond1}, \eqref{cond2}. We get
\begin{align*}\mathcal{I}_v(n_2, m_1, m_2, d_1, d_2, c)=&\frac{e_c\Big(a_1v+\frac{m_1}{2}\big(v+m_2-\sqrt{(v+m_2)^2+4vn_2}\big)\Big)e^{-\pi i/4}\sqrt{c}}{\sqrt{|f''(x_v^*(n_2))|}\cdot x_v^{*}(n_2)^{3/2}\cdot  n_2^{3/2}} \hbar_{d_1, d_2}^*\Big(\frac{cK^2}{8\pi  n_2x_v^*(n_2)}\Big)\\&+ O(K^{-2+\varepsilon}).
\end{align*} 
This completes the proof of the lemma.
\end{proof}
We expect the main term of \eqref{maintermanderrorterm} to come from the $0$-frequency, i.e. when $v=0$. We have
\begin{lem}\label{maintermlemma} Let $\mathcal{I}_v(n_2, m_1, m_2, d_1, d_2, c)$ be defined as in \eqref{integral}. Then
$$ \sum_{n_2 \equiv a_2\mod{c}} \mathcal{I}_0(n_1, m_1, m_2, d_1, d_2, c) =  e^{-\pi i/4}\sqrt{2c} \sum_{n_2\equiv a_2\mod{c}} \frac{1}{m_1n_2^2} \hbar^*\Big(\frac{cK^2m_2}{8\pi n_2^2m_1}\Big) +O(K^{-1+\varepsilon}).$$
\end{lem}
\begin{proof}
This lemma is a direct consequence of Lemma \ref{StationaryInt}, specialized to the case $v=0$. For $v=0$ the stationary point is given by $x_0^*(n_2)=\frac{m_1n_2}{m_2}$. From \eqref{secondderivative} we see that
$$\sqrt{|f''\big(x_0^*(n_2)\big)|}=\frac{m_2^3}{2m_1n_2(n_2+m_2)} = \frac{m_2^3}{2m_1n_2^2}\cdot\Big(1+O(K^{-1/2+\varepsilon})\Big).$$
Moreover, a direct computation gives
$$e_c\big(f(x_0^*(n_2), n_2,m_1,m_2)\big)=e_c(0)=1$$ 
and thus the claimed result.

\end{proof}
On the other hand, the contribution of the non-zero frequencies, i.e. when $v \neq 0$, is negligible. To show this we will need to detect further cancellation in the summation over $n_2$. We have
\begin{lem}\label{errortermlemma}Let $\mathcal{I}_v(n_2, m_1, m_2, d_1, d_2, c)$ be defined as in \eqref{integral}. Then
\begin{equation}\label{errorterm2}\sum_{n_2 \equiv a_2\mod{c}} \sum_{v\neq 0} \mathcal{I}_v(n_2, m_1, m_2, d_1, d_2, c) \ll \frac{d_2^2}{\sqrt{c} K^{1-\varepsilon}}.
\end{equation}
\end{lem}
\begin{proof}
Let $ET$ (for error term) denote the left-hand side of \eqref{errorterm2}. First we use the Poisson summation formula for the summation over $n_2$. Together with Lemma \ref{StationaryInt} it follows that
\begin{align}\label{secondstationary}
ET=\frac{e^{-\pi i/4}}{c}\sum_{w} \sum_{v\neq 0}\int_{-\infty}^\infty& e_c\Big(a_1v+\frac{m_1}{2}\big(v+m_2-\sqrt{(v+m_2)^2+4vy}\big)-w(y-a_2)\Big)\\
&\times \frac{\sqrt{c}}{\sqrt{|f''\big(x_v^*(y))|}}\frac{1}{ x_v^{*}(y)^{3/2} y^{3/2}}\hbar^*\Big(\frac{cK^2}{8\pi x_v^*(y) y}\Big) dy +O(K^{-2+\varepsilon})
\end{align}
with 
\begin{equation}\label{stationarypointy} x_v^*(y)=\frac{m_1\big(m_2+2y+v-\sqrt{(v+m_2)^2+4vy}\big)}{2\sqrt{(v+m_2)^2+4vy}} \asymp \frac{K}{d_1}.
\end{equation}
As in Lemma \eqref{StationaryInt} we perform a stationary phase analysis on the integral
$$\mathcal{J}_w (m_1, m_2, v):=\int_{-\infty}^\infty e_c\big(g_{m_1, m_2, v}(y) - w(y-a_2)\big)\frac{\sqrt{c}}{\sqrt{|f''\big(x_v^*(y))|}}\frac{1}{ x_v^{*}(y)^{3/2} y^{3/2}}\hbar^*\Big(\frac{cK^2}{8\pi x_v^*(y) y}\Big) dy$$
with
$$g_{m_1, m_2, v}(y)=a_1v+\frac{m_1}{2}\big(v+m_2-\sqrt{(v+m_2)^2+4vy}\big).$$

It is again useful to have the critical cases in mind when $y \sim K$ and $m_1, m_2 \sim \sqrt{K}$. The first derivative $g'_{m_1, m_2, v}(y)$ is then again roughly bounded, while $g''_{m_1, m_2, v}(y)$ is of size $1/K$. Consequently, we again expect square root cancellation in $n_2$ (see for example \cite[Corollary 8.12]{IwaniecKowalskiAnalytica}). We now perform the stationary phase method on the integral over $y$ of expression $ET$. The stationary points are given by the solutions to the equation 
$$\frac{\partial}{\partial y} \Big(g_{m_1, m_2, v}(y) - (w-a_2)\Big) = - \frac{m_1 v}{\sqrt{(v+m_2)^2+4vy}}-w =0.$$
Since $\frac{\partial}{\partial y} g_{m_1, m_2, v}(y) \ll K^\varepsilon$, we also see that $w \ll K^\epsilon$. Moreover, we have the following bounds on the higher derivatives:
\begin{equation}\label{derivativesg} \frac{\partial^2}{\partial y^2} g_{m_1, m_2, v}(y)=\frac{4 m_1 v^2}{\big((v+m_2)^2+4vy\big)^{3/2}}\quad \text{and} \quad \frac{\partial^j}{\partial y^j} g_{m_1, m_2, v}(y)\ll_j \frac{m_1 v^j}{\big((v+m_2)^2+4vy\big)^{j-1/2}},
\end{equation}
for $j\geq 3$.
From the computations in the proof of Lemma \ref{StationaryInt}, in particular equation \eqref{stationarypoint} and \eqref{secondderivative} we have
$$\frac{\sqrt{c}}{\sqrt{|f''\big(x_v^*(y)\big)|}} \frac{1}{x_v^{*}(y)^{3/2} y^{3/2}}=\frac{\sqrt{2c}}{m_1y^2}\cdot\big(1+O(K^{-1/2+1/32+\varepsilon})\big).$$
Similarly to \eqref{derivatives} we will need bounds on the derivatives of the involved weight function with respect to $y$. It is useful to first compute $$\frac{\partial^j}{\partial{y^j}}\big( x_v^*(y)\big)\ll_j\frac{m_1 v^{j-1}(m_2^2+vy)}{(4vy+(m_2+v)^2)^{1/2+j}}\ll_j \frac{K}{d_1}\Big(\frac{1}{y}\Big)^{j}.$$
Using the chain rule, a similar computation as in \eqref{derivatives} yields
\begin{equation}\label{derivatives2} \frac{\partial^j}{\partial y^j} \frac{\sqrt{2c}}{m_1 y^2} \hbar^*\Big(\frac{cK^2}{8\pi x_v^*(y)\cdot y}\Big) \ll_j \frac{\sqrt{c}}{m_1 y^2}\cdot K \Big(\frac{K^\epsilon}{y}\Big)^j \ll_j \frac{\sqrt{c}}{m_1} \frac{d_2^2}{K} \cdot \Big(\frac{d_2}{K^{1-\varepsilon}}\Big)^j.
\end{equation}
We can now again establish the various required quantities for Lemma \ref{stationaryphase}. From the computation \eqref{derivatives2} we can see that $X=\sqrt{c} d_2^2 /(m_1 K)$ 
and $V=K^{1-\varepsilon}/d_2$. From the bounds \eqref{derivativesg} we find $Y=m_1\cdot \sqrt{(v+m_2)^2+4vK/d_2}$ and $Q=((v+m_2)^2+4vK/d_2)/v$. The trivial bound of Lemma \ref{stationaryphase} yields
\begin{equation}\label{StationaryInt2}\mathcal{J}_w(m_1, m_2, v) \ll \frac{XQ}{\sqrt{Y}} \ll \frac{\sqrt{c}d_2^2}{v K} \cdot \bigg( \frac{\sqrt{(v+m_2)^2+4vK/d_2}}{m_1}\bigg)^{3/2} \ll \frac{\sqrt{c}d_2^2}{v K}.
\end{equation}
Here we used that $\frac{\sqrt{(v+m_2)^2+4vK/d_2}}{m_1}\asymp 1$, which can be deduced for example from the size of the stationary point (see \eqref{stationarypointy}).

From the bounds $w\ll K^\varepsilon$, $v \ll K^\varepsilon$ and \eqref{StationaryInt2} it follows that
$$ET \ll \frac{d_2^2}{\sqrt{c} K^{1-\varepsilon}}.$$
This concludes the proof of the lemma.
\end{proof}
Using the previous lemmas we will obtain the following formula for the off-diagonal $\mathcal{OD}$:
\begin{lem}\label{smoothsums}
Let $\mathcal{OD}$ be defined as in \eqref{offdiagonal2}. Then
\begin{align*}
\mathcal{OD}=&\frac{- \sqrt{2\pi} K^2}{8}  \sum_{c} \frac{\phi(c)}{c}  \sum_{d_1,d_2} \sum_{n_2} \sum_{m_1, m_2} \frac{1}{d_1d_2m_1n_2^2} \cdot \\
&\times\int_0^\infty \frac{h(\sqrt{u})\sqrt{u} }{\sqrt{2\pi u}}\psi_1\Big(\frac{\sqrt{u}Km_2}{4\pi m_1 d_1 n_2}\Big)\psi_2\Big(\frac{\sqrt{u}K}{4\pi d_2n_2}\Big)\exp\Big(-\frac{\sqrt{u}Km_2^2}{4n_2^2}\Big)\cos\Big( u \frac{cK^2m_2}{8\pi m_1n_2^2}\Big)du  \nonumber \\
&+O(K^{5/4}) \nonumber
\end{align*} 
\end{lem}
\begin{proof}
Combining Lemma \ref{afterPoisson} and Lemma \ref{maintermlemma} we obtain the main term of the off-diagonal $\mathcal{OD}$ given by
\begin{align*}\label{remain}
-\sqrt{2\pi }K \Im \bigg(&e^{\pi i/2}  \sum_{c} \sum_{\substack{b_1 \mod{c}, \\ b_2 \mod{c}}}\sum_{\substack{a_1 \mod{c}, \\ a_2 \mod{c}}}\frac{S(a_1(a_1+b_1),n_2(n_2+b_2);c)}{c}e_c(2a_1a_2+a_1b_2+a_2b_1) \cdot \\
& \times \sum_{d_1,d_2} \sum_{\substack{m_1 \equiv b_1 \mod{c}\\m_2 \equiv b_2 \mod{c}}}\sum_{n_2\equiv a_2\mod{c}}\frac{1}{d_1d_2 m_1 n_2^2}\hbar_{d_1,d_2}^*\Big(\frac{c K^2 m_2}{8\pi m_1n_2^2}\Big)\bigg) 
 \end{align*}
with
 \begin{align*}
 &\hbar_{d_1,d_2}^*\Big(\frac{cK^2m_2}{8\pi m_1 n_2^2}\Big) =\\
 &\frac{K}{8}\int_0^\infty \frac{h(\sqrt{u})\sqrt{u}}{\sqrt{2\pi u}} \psi_1\Big(\frac{\sqrt{u}Km_2}{2\pi m_1 d_1 (2n_2+m_2)}\Big)\psi_2\Big(\frac{\sqrt{u}K}{2\pi d_2(2n_2+m_2)}\Big)\exp\Big(-\frac{\sqrt{u}Km_2^2}{(2n_2+m_2)^2}\Big)e^{i u \frac{cK^2m_2}{m_1n_2^2}} du.
 \end{align*}
On the other hand, by Lemma \ref{afterPoisson} and Lemma \ref{errortermlemma}, the error term is bounded by 
\begin{align*}
K^{\varepsilon} \sum_{c} \sum_{\substack{b_1 \mod{c}, \\ b_2 \mod{c}}}\sum_{\substack{a_1 \mod{c}, \\ a_2 \mod{c}}}\frac{|S(a_1(a_1+b_1),n_2(n_2+b_2);c)|}{c^2}\sum_{d_1,d_2} \sum_{\substack{m_1 \equiv b_1 \mod{c}\\m_2 \equiv b_2 \mod{c}}}\frac{d_2}{d_1} \ll K^{5/4}.
\end{align*}
We thus have
\begin{align*}
\mathcal{OD}=-\sqrt{2\pi }K \Im \bigg(&e^{\pi i/2}  \sum_{c} \sum_{\substack{b_1 \mod{c}, \\ b_2 \mod{c}}}\sum_{\substack{a_1 \mod{c}, \\ a_2 \mod{c}}}\frac{S(a_1(a_1+b_1),n_2(n_2+b_2);c)}{c}e_c(2a_1a_2+a_1b_2+a_2b_1) \cdot \\
& \times \sum_{d_1,d_2} \sum_{\substack{m_1 \equiv b_1 \mod{c}\\m_2 \equiv b_2 \mod{c}}}\sum_{n_2\equiv a_2\mod{c}}\frac{1}{d_1d_2 m_1 n_2^2}\hbar_{d_1,d_2}^*\Big(\frac{c K^2 m_2}{8\pi m_1n_2^2}\Big)\bigg) +O(K^{5/4}).
 \end{align*}
Note that $\frac{cK^2m_2}{m_1n_2^2} \ll K^\varepsilon$ as otherwise we have rapid decay (using integration by parts). Additionally, we have the derivative bounds
\begin{equation}\label{boundn2}\frac{\partial^j}{\partial n_2^j} \hbar^*\Big(\frac{cK^2m_2}{8\pi m_1 n_2^2}\Big) \ll K\cdot  \Big(\frac{K^\epsilon}{n_2}\Big)^j \ll K \cdot K^{-j/2},
\end{equation}
\begin{equation}\label{boundm1}\frac{\partial^j}{\partial m_1^j} \hbar^*\Big(\frac{cK^2m_2}{8\pi m_1 n_2^2}\Big) \ll K \cdot \Big(\frac{K^\epsilon}{m_1}\Big)^j \ll K \cdot K^{-j/16}
\end{equation}
and
\begin{equation}\label{boundm2}\frac{\partial^j}{\partial m_2^j} \hbar^*\Big(\frac{cK^2m_2}{8\pi m_1 n_2^2}\Big) \ll K \cdot\Big(\frac{K}{m_1d_1n_2}+\frac{Km_2}{n_2^2}+\frac{cK^2}{m_1n_2^2}\Big)^j \ll K\cdot K^{-j/16}.
\end{equation}
 With the Poisson summation formula we see that 
 $$\sum_{n_2\equiv a_2\mod{c}} \frac{1}{d_1d_2 m_1 n_2^2}\hbar^*_{d_1, d_2}\Big(\frac{c K^2 m_2}{8\pi m_1n_2^2}\Big)=\frac{1}{c} \sum_u \int_{-\infty}^{\infty} \frac{1}{d_1d_2 m_1 x^2}\hbar^*_{d_1, d_2}\Big(\frac{c K^2 m_2}{8\pi m_1x^2}\Big) e_c\big(-u(x-a_2)\big)dx.$$
 Repeated integration by parts and the bounds \eqref{boundn2} show that the non-zero frequencies are bounded by $O_j(K^{-10})$. For the $0$-frequency we have
 $$\frac{1}{c} \int_{-\infty}^{\infty} \frac{1}{d_1d_2 m_1 x^2}\hbar^*_{d_1, d_2}\Big(\frac{c K^2 m_2}{8\pi m_1x^2}\Big) dx =\frac{1}{c}\sum_{n_2} \frac{1}{d_1d_2 m_1 n_2^2}\hbar^*_{d_1, d_2}\Big(\frac{c K^2 m_2}{8\pi m_1n^2}\Big)+O_j(K^{-10})$$
 again by the Poisson summation formula. We proceed in the same way for the summation over $m_1$ and $m_2$, using the derivative bounds \eqref{boundm1}, \eqref{boundm2} respectively, to bound the non-zero frequencies. It follows that the off-diagonal $\mathcal{OD}$ is up to a negligible error term given by
\begin{align*}-\sqrt{2\pi} K\Im\bigg(e^{-\pi i/2}&\sum_{c} \sum_{\substack{a_1 \mod{c}, \\ a_2 \mod{c}}} \sum_{\substack{b_1 \mod{c}, \\ b_2 \mod{c}}} \frac{S(a_1(a_1+b_1),a_2(a_2+b_2);c)}{c^4}e_c(2a_1a_2+a_1b_2+a_2b_1) \cdot \\
& \times \sum_{d_1,d_2} \sum_{n_2}\sum_{m_1, m_2} \frac{1}{d_1d_2m_1n_2^2} \hbar_{d_1,d_2}^*\Big(\frac{c K^2 m_2}{8\pi m_1n_2^2}\Big)\bigg).
\end{align*}
We now use Lemma \ref{Kloostermansum} to simplify the summation of the Kloosterman sum over arithmetic progressions and see that
\begin{align*}\label{offdiagonalcomplete1}&\mathcal{OD}=\frac{- \sqrt{2\pi} K^2}{8} \Im\bigg( e^{-\pi i/2} \sum_{c} \frac{\phi(c)}{c}  \sum_{d_1,d_2} \sum_{n_2} \sum_{m_1, m_2} \frac{1}{d_1d_2m_1n_2^2} \hbar_{d_1,d_2}^*\Big(\frac{c K^2 m_2}{8\pi m_1n_2^2}\Big)\bigg)+O(K^{5/4})\nonumber
\end{align*} 
By a Taylor expansion, using \eqref{cond1} and \eqref{cond2}, we see that
\begin{align*}\hbar_{d_1,d_2}^*\Big(\frac{cK^2m_2}{8\pi m_1 n_2^2}\Big)= &\int_0^\infty \frac{h(\sqrt{u})\sqrt{u} }{\sqrt{2\pi u}}\psi_1\Big(\frac{\sqrt{u}Km_2}{4\pi m_1 d_1 n_2}\Big)\psi_2\Big(\frac{\sqrt{u}K}{4\pi d_2n_2}\Big)\exp\Big(-\frac{\sqrt{u}Km_2^2}{4n_2^2}\Big)e^{iu \frac{cK^2m_2}{8\pi m_1n_2^2}}du  \\
&+ O(K^{1/2+\varepsilon}),
\end{align*}
so that the lemma follows upon evaluating the imaginary part of $\mathcal{OD}$.
\end{proof}

Finally, we are in the position to evaluate the off-diagonal $\mathcal{OD}$ asymptotically. To do so we will relate $\mathcal{OD}$ to a complex contour integral over several variables. We then evaluate this contour integral with the residue theorem.
\begin{lem}\label{offdiagonalfinaloutcome} Let $\mathcal{OD}$ be given by expression \eqref{offdiagonal2}. We have
$$\mathcal{OD}=K^{3/2} \cdot \int_0^\infty \frac{h(\sqrt{u})u^{1/4}}{\sqrt{2\pi u}}du \cdot \frac{\sqrt{2}\pi}{16}\frac{1}{2\pi i} \int_{(\varepsilon)} \tilde{\psi_1}(s_4)\tilde{\psi_2}(s_4)\zeta(1+s_4)\zeta(1-s_4)ds_4+O(K^{5/4+\epsilon}).$$
\end{lem}
\begin{proof}
From Lemma \ref{smoothsums} we see that
\begin{align*}
\mathcal{OD}=&\frac{- \sqrt{2\pi} K^2}{8}  \sum_{c} \frac{\phi(c)}{c}  \sum_{d_1,d_2} \sum_{n_2} \sum_{m_1, m_2} \frac{1}{d_1d_2m_1n_2^2} \cdot \\
&\times\int_0^\infty \frac{h(\sqrt{u})\sqrt{u} }{\sqrt{2\pi u}}\psi_1\Big(\frac{\sqrt{u}Km_2}{4\pi m_1 d_1 n_2}\Big)\psi_2\Big(\frac{\sqrt{u}K}{4\pi d_2n_2}\Big)\exp\Big(-\frac{\sqrt{u}Km_2^2}{4n_2^2}\Big)\cos\Big( u \frac{cK^2m_2}{8\pi m_1n_2^2}\Big)du .\nonumber
\end{align*}
To evaluate this expression asymptotically we perform an inverse Mellin transform on $\psi_1, \psi_2$ and the exponential function. 
Then $\mathcal{OD}$ is equal to
\begin{align}
\frac{\sqrt{2\pi} K^2}{8}& \sum_{c} \frac{\phi(c)}{c}  \sum_{d_1,d_2} \sum_{n_2} \sum_{m_1, m_2} \frac{1}{d_1d_2m_1n_2^2} \cdot \\
 &\times \frac{1}{(2\pi i)^3} \int_{(1/2+\varepsilon)}\int_{(2)}\int_{(1+\varepsilon)} \tilde{\psi_1}(s_1)\tilde{\psi_2}(s_2)\Gamma(s_3) \Big(\frac{Km_2}{4\pi m_1d_1n_2}\Big)^{s_1} \Big(\frac{K}{4\pi d_2n_2}\Big)^{s_2} \Big(\frac{4n_2^2}{Km_2^2}\Big)^{s_3} \cdot\nonumber \\ 
 & \times \int_0^\infty \frac{h(\sqrt{u})}{\sqrt{2\pi u}}u^{(1+s_1+s_2-s_3)/2}\cos\Big(u \frac{cK^2m_2}{8\pi m_1n_2^2}\Big) du ds_1 ds_2 ds_3\nonumber
\end{align}
Finally, we also perform an inverse Mellin transform on 
$$\hbar_{1+s_1+s_2-s_3}^{\Re}(v):=\int_0^\infty \frac{h(\sqrt{u})}{\sqrt{2\pi u}} u^{(1+s_1+s_2-s_3)/2} \cos(uv) du$$ 
as indicated in equation \eqref{inverseMellin}. We arrive at
\begin{align*}
\mathcal{OD}=\frac{\sqrt{2\pi} K^2}{8} &\frac{1}{(2\pi i)^4} \int_{(1+\varepsilon)}\int_{(1/2+3\varepsilon)}\int_{(2)}\int_{(1+3\varepsilon)} \tilde{\psi_1}(s_1)\tilde{\psi_2}(s_2)\Gamma(s_3)\tilde{\hbar}_{1+s_1+s_2-s_3}^{\Re}(s_4)\cdot \\
 & \times \sum_{d_1,d_2} \sum_{n_2} \sum_{m_1, m_2} \frac{1}{d_1d_2m_1n_2^2}\cdot \Big(\frac{\sqrt{u}Km_2}{4\pi m_1d_1n_2}\Big)^{s_1} \Big(\frac{\sqrt{u}K}{4\pi d_2n_2}\Big)^{s_2} \Big(\frac{4n_2^2}{\sqrt{u}Km_2^2}\Big)^{s_3}\nonumber\\ 
& \times{\sum_{c}} \frac{\phi(c)}{c}\Big(\frac{8\pi m_1n_2^2}{cK^2m_2}\Big)^{s_4}du ds_1ds_2ds_3ds_4. \nonumber
\end{align*}
We were allowed to interchange the order of summation and integration by the absolute convergence of the integrand in the given ranges. We now rewrite the various summations in terms of zeta functions and get
\begin{align}\label{offdiagonalcomplete3}
\mathcal{OD}=\frac{\sqrt{2\pi} K^2}{8}&\frac{1}{(2\pi i)^4} \int_{(1+\varepsilon)}\int_{(1/2+3\varepsilon)}\int_{(2)}\int_{(1+3\varepsilon)} \tilde{\psi_1}(s_1)\tilde{\psi_2}(s_2)\Gamma(s_3)\tilde{\hbar}^{\Re}_{1+s_1+s_2-s_3}(s_4)\cdot \\
&\times \zeta(1+s_1)\zeta(1+s_2)\zeta(2+s_1+s_2-2s_3-2s_4)\zeta(1+s_1-s_4)\zeta(-s_1+2s_3+s_4)  \cdot \nonumber \\
&\times \zeta(s_4)/\zeta(1+s_4)(4\pi)^{-s_1-s_2}4^{s_3}(8\pi)^{s_4} K^{s_1+s_2-s_3-2s_4} ds_1d_2ds_3ds_4.\nonumber
\end{align}
The zeta functions $\zeta(1+s_1)$ and $\zeta(1+s_2)$ arise from summing over $d_1, d_2$ respectively. The summation over $n_2$ yields $\zeta(2+s_1+s_2-2s_3-2s_4)$, while summing over $m_1$ gives rise to $\zeta(1+s_1-s_4)$. The $m_2$-variable leads to the factor $\zeta(-s_1+2s_2+s_4)$ and finally the summation over $c$ gives $\zeta(s_4)/\zeta(1+s_4)$. Here we used that $\sum_{c}\frac{\phi{c}}{c^s} =\frac{\zeta(s-1)}{\zeta(s)}$ for $Re(s)>2$. To evaluate expression $\mathcal{OD}$ asymptotically we will iteratively shift the contours and pick up poles. 

We start to compute the contour integral \eqref{offdiagonalcomplete3} by shifting the line from $\Re(s_2)=2$ to $\Re(s_2)=-100$. We pick up a simple pole at $s_2=0$ and $s_2=-1-s_1+2s_3+2s_4$. The new line integral is negligible by the rapid decay of $\tilde{\psi_1}, \tilde{\psi_2}, \tilde{\hbar}$ and the Gamma function. The residue at $s_2=0$ is given by
\begin{align}\label{offdiagonalcomplete5}
\frac{\sqrt{2\pi} K^2}{8}&\frac{1}{(2\pi i)^3} \int_{(1+\varepsilon)}\int_{(1/2+3\varepsilon)}\int_{(1+3\varepsilon)} \tilde{\psi_1}(s_1)\tilde{\psi_2}(0)\Gamma(s_3)\tilde{\hbar}^{\Re}_{1+s_1-s_3}(s_4)\cdot \\
&\times \zeta(1+s_1)\zeta(2+s_1-2s_3-2s_4)\zeta(1+s_1-s_4)\zeta(-s_1+2s_3+s_4)\zeta(s_4)/\zeta(1+s_4)  \cdot \nonumber \\
&\times (4\pi)^{-s_1}4^{s_3}(8\pi)^{s_4} K^{s_1-s_3-2s_4} ds_1ds_3ds_4.\nonumber
\end{align}
Moving the line $\Re(s_1)=1+2\varepsilon$ to $\Re(s_1)=-100$ yields poles at $s_1=0$ and $s_1=s_4$. First, we consider the residue at $s_1=0$, which is given by
\begin{align}\label{offdiagonalcomplete6}
\frac{\sqrt{2\pi} K^2}{8}&\frac{1}{(2\pi i)^2} \int_{(1+\varepsilon)}\int_{(1/2+2\varepsilon)} \tilde{\psi_1}(0)\tilde{\psi_2}(0)\Gamma(s_3)\tilde{\hbar}^{\Re}_{1-s_3}(s_4)\cdot \\
&\times \zeta(2-2s_3-2s_4)\zeta(1-s_4)\zeta(2s_3+s_4)\zeta(s_4)/\zeta(1+s_4)  4^{s_3}(8\pi)^{s_4} K^{-s_3-2s_4} ds_3ds_4.\nonumber
\end{align}
Expression \eqref{offdiagonalcomplete6} is negligible upon shifting $\Re(s_3)=1/2+3\varepsilon$ to $\Re(s_3)=100$ and the rapid decay of $\tilde{\hbar}$ and the Gamma function. On the other hand the residue of the pole at $s_1=s_4$ of \eqref{offdiagonalcomplete5} leads to
\begin{align}\label{offdiagonalcomplete7}
\frac{\sqrt{2\pi} K^2}{8}&\frac{1}{(2\pi i)^2} \int_{(1+\varepsilon)}\int_{(1/2+2\varepsilon)}\tilde{\psi_1}^{\Re}(s_4)\tilde{\psi_2}(0)\Gamma(s_3)\tilde{\hbar}^{\Re}_{1+s_4-s_3}(s_4)\cdot \\
&\times \zeta(2-2s_3-s_4)\zeta(2s_3)\zeta(s_4)  (4\pi)^{-s_4}4^{s_3}(8\pi)^{s_4} K^{-s_3-s_4} ds_3ds_4.\nonumber
\end{align}
This integral is again negligible after sending $\Re(s_3)=1/2+3\varepsilon$ to $\Re(s_3)=100$, as we pick up no poles and we can bound everything trivially. In total we found that the contribution from poles that arise after $s_2=0$ is negligible. We now evaluate the residue of  \eqref{offdiagonalcomplete5} at the pole $s_2=-1-s_1+2s_3+2s_4$. We get
\begin{align}\label{offdiagonalcomplete8}
\frac{\sqrt{2\pi} K^2}{8}&\frac{1}{(2\pi i)^3} \int_{(1+\varepsilon)}\int_{(1/2+3\varepsilon)}\int_{(1+3\varepsilon)} \tilde{\psi_1}(s_1)\tilde{\psi_2}(-1-s_1+2s_3+2s_4)\Gamma(s_3)\tilde{\hbar}^{\Re}_{s_3+2s_4}(s_4)\cdot \\
&\times \zeta(1+s_1)\zeta(-s_1+2s_3+2s_4)\zeta(1+s_1-s_4)\zeta(-s_1+2s_3+s_4)\zeta(s_4)/\zeta(1+s_4)  \cdot \nonumber \\
&\times (4\pi)^{1-2s_3-2s_4}4^{s_3}(8\pi)^{s_4} K^{-1+s_3} ds_1ds_3ds_4.\nonumber
\end{align}
Next we move the line $\Re(s_3)=1/2+3\varepsilon$ to $\Re(s_3)=\varepsilon$ (stopping before the pole of the Gamma function) and capture poles at $s_3=1/2+s_1/2-s_4$ and $s_3=1/2+s_1/2-s_4/2$. The new line integrals contribute at most $O(K^{1+\varepsilon})$. The residue at $s_3=1/2+s_1/2-s_4$ is given by
\begin{align}\label{offdiagonalcomplete9}
\frac{\sqrt{2\pi} K^2}{8}&\frac{1}{(2\pi i)^2} \int_{(1+\varepsilon)}\int_{(1+3\varepsilon)} \tilde{\psi_1}(s_1)\tilde{\psi_2}(0)\Gamma(1/2+s_1/2-s_4)\tilde{\hbar}^{\Re}_{1/2+s_1/2+s_4}(s_4)\cdot \\
&\times \zeta(1+s_1)\frac{1}{2}\zeta(1+s_1-s_4)\zeta(1-s_4)\zeta(s_4)/\zeta(1+s_4) \cdot \nonumber\\
&\times (4\pi)^{-s_1}4^{1/2+s_1/2-s_4}(8\pi)^{s_4} K^{-1/2+s_1/2-s_4} ds_1ds_4.\nonumber
\end{align}
We then shift $\Re(s_1)=1+3\varepsilon$ to $\Re(s_1)=-1+\varepsilon$ and pick up simple poles at $s_1=0$, $s_1=s_4$ from the zeta functions and the simple pole $s_1=-1+2s_4$ from the Gamma function. The new line integral is clearly negligible. The residue from the gamma function contributes at most $O(K)$. The residue at $s_1=0$, given by
\begin{align}\label{offdiagonalcomplete10}
\frac{\sqrt{2\pi} K^2}{8}&\frac{1}{(2\pi i)^2} \int_{(1+\varepsilon)}\tilde{\psi_1}(0)\tilde{\psi_2}(0)\Gamma(1/2-s_4)\tilde{\hbar}^{\Re}_{1/2+s_4}(s_4)\cdot \\
&\times \frac{1}{2}\zeta(1-s_4)^2\zeta(s_4)/\zeta(1+s_4) 4^{1/2-s_4}(8\pi)^{s_4} K^{-1/2-s_4}ds_4,\nonumber
\end{align}
is clearly also negligible. At the pole $s_1=s_4$ the residue is given by
\begin{align}\label{offdiagonalcomplete11}
\frac{\sqrt{2\pi} K^2}{8}&\frac{1}{(2\pi i)^2} \int_{(1+\varepsilon)} \tilde{\psi_1}(s_4)\tilde{\psi_2}(0)\Gamma(1/2-s_4/2)\tilde{\hbar}^{\Re}_{1/2+3/2s_4}(s_4)\cdot \\
&\times \frac{1}{2}\zeta(1-s_4)\zeta(s_4) (4\pi)^{-s_4}4^{1/2-s_4/2}(8\pi)^{s_4} K^{-1/2-s_4/2} ds_4,\nonumber
\end{align}
which again contributes only to the error term. It remains to compute the chain of residues of \eqref{offdiagonalcomplete8} starting with $s_3=1/2+s_1/2-s_4/2$. At this point we get
\begin{align}\label{offdiagonalcomplete12}
\frac{\sqrt{2\pi} K^2}{8}&\frac{1}{(2\pi i)^2} \int_{(1+\varepsilon)}\int_{(1+3\varepsilon)} \tilde{\psi_1}(s_1)\tilde{\psi_2}(s_4)\Gamma(1/2+s_1/2-s_4/2)\tilde{\hbar}^{\Re}_{1/2+s_1/2+3/2s_4}(s_4)\cdot \\
&\times \zeta(1+s_1)\zeta(1+s_1-s_4)\frac{1}{2}\zeta(s_4)  (4\pi)^{-s_1-s_4}4^{1/2 +s_1/2-s_4/2}(8\pi)^{s_4} K^{-1/2+s_1/2-s_4/2} ds_1ds_4.\nonumber
\end{align}
We then shift $\Re(s_1)=1+3\varepsilon$ to $\Re(s_1)=2\varepsilon$ and pick up a simple pole $s_1=s_4$. The new line integral is bounded by $O(K^{1+\epsilon})$ and is therefore negligible. Our expected main term, the residue of the pole $s_1=s_4$, is given by
\begin{align}\label{offdiagonalcomplete13}
&\frac{\sqrt{2\pi} K^{3/2}}{8} \frac{1}{2\pi i} \int_{(1+\varepsilon)}\tilde{\psi_1}(s_4)\tilde{\psi_2}(s_4)\Gamma(1/2)\tilde{\hbar}^{\Re}_{1/2+2s_4}(s_4) \zeta(1+s_4)\zeta(s_4) (2\pi)^{-s_4} ds_4.
\end{align}
We now shift the line $\Re(s_4)=1+\varepsilon$ to $\Re(s_4)=\varepsilon$ to simplify this expression. Note that the residue of the pole at $s_4=1$ is $0$, since $\tilde{\hbar}^{\Re}_{3/2}(1)=0$. On the new line $\Re(s_4)=\varepsilon$ we can explicitly evaluate $\tilde{\hbar}^{\Re}_{1/2+2s_4}(s_4)$, leading to 
\begin{align}\label{offdiagonalcomplete14}
&K^{3/2} \cdot \int_0^\infty \frac{h(\sqrt{u})u^{1/4}}{\sqrt{2\pi u}} du \cdot \frac{\sqrt{2}\pi}{8} \frac{1}{2\pi i} \int_{(\varepsilon)}\tilde{\psi_1}(s_4)\tilde{\psi_2}(s_4)\Gamma(s_4)\cos(\pi s/2) \zeta(1+s_4)\zeta(s_4) (2\pi)^{-s_4} ds_4.
\end{align}
Finally, we use the functional equation
$$\zeta(1-s)=2(2\pi)^{-s}\cos(\pi s/2)\Gamma(s)\zeta(s)$$ so that the off-diagonal is up to an error term of size $O(K^{5/4+\varepsilon})$ equal to
$$K^{3/2} \cdot \int_0^\infty \frac{h(\sqrt{u})u^{1/4}}{\sqrt{2\pi u}}du \cdot \frac{\sqrt{2}\pi}{16}\frac{1}{2\pi i} \int_{(\varepsilon)} \tilde{\psi_1}(s_4)\tilde{\psi_2}(s_4)\zeta(1+s_4)\zeta(1-s_4)ds_4.$$
This matches exactly with term \eqref{diagonalstep6} from the diagonal, if we suppose that $\psi_1(y)$ is even, i.e. $\psi_1(y)=\psi_1(1/y)$ and thus $\tilde{\psi_1}(s)=\tilde{\psi_1}(-s)$.
\end{proof}

\subsection{Proof of the main theorem}
\begin{proof}[Proof of Theorem \ref{mainthm}]
Recall the definition of $\mathcal{E}_{\psi}$ (se \eqref{errorterm}) and $S_{\psi}$ (see \eqref{shiftedconvolutiondefi}). The work of Luo and Sarnak (\cite[Section 5]{luoandsarnakMassEquidistributionHecke}) shows that
\begin{equation}\label{LuoSarnak}\sum_{k\equiv 0\mod{2}} h\Big(\frac{k-1}{K}\Big) \sum_{f\in H_k} L(1, \sym^2 f) |E_\psi|^2 \ll K^{1+\varepsilon}.
\end{equation}
Moreover, we have
\begin{align*}
V(\psi_1, \psi_2) &=\sum_{k\equiv 0\mod{2}} h\Big(\frac{k-1}{K}\Big) \sum_{f\in H_k} L(1, \sym^2 f) (S_{\psi_1} + E_{\psi_1})(S_{\psi_2}+E_{\psi_2})\\
&=\sum_{k\equiv 0\mod{2}} h\Big(\frac{k-1}{K}\Big) \sum_{f\in H_k} L(1, \sym^2 f) (S_{\psi_1}S_{\psi_2}+S_{\psi_1}E_{\psi_2}+S_{\psi_2}E_{\psi_1}+E_{\psi_1}E_{\psi_2}).
\end{align*}
We evaluated the main term
$$\mathcal{M}(\psi_1, \psi_2)=\sum_{k\equiv 0\mod{2}} h\Big(\frac{k-1}{K}\Big) \sum_{f\in H_k} L(1, \sym^2 f) S_{\psi_1}S_{\psi_2}=\mathcal{D}+\mathcal{OD},$$ in Section \ref{VarianceComputation} with Lemma \ref{diagonal} and Lemma \ref{offdiagonalfinaloutcome}. Theorem \ref{mainthm} follows then from the Cauchy-Schwarz ineqality and the bound \eqref{LuoSarnak}.
\end{proof}
\bibliographystyle{alpha}
\bibliography{Bib}

\newcommand{\etalchar}[1]{$^{#1}$}
\begin{thebibliography}{CFK{\etalchar{+}}05}

\bibitem[BKY13]{blomerDistributionMassHolomorphic2013b}
Valentin Blomer, Rizwanur Khan, and Matthew Young.
\newblock Distribution of mass of holomorphic cusp forms.
\newblock {\em Duke Mathematical Journal}, 162(14):2609--2644, 2013.

\bibitem[CFK{\etalchar{+}}05]{conreyIntegralMomentsLfunctions2002}
J.~B. Conrey, D.~W. Farmer, J.~P. Keating, M.~O. Rubinstein, and N.~C. Snaith.
\newblock Integral moments of {{L}}-functions.
\newblock {\em Proceedings of the London Mathematical Society}, 91(1):33--104,
  2005.

\bibitem[CTZ13]{christiansonQuantumErgodicRestriction2012}
Hans Christianson, John~A. Toth, and Steve Zelditch.
\newblock Quantum ergodic restriction for {C}auchy data: interior que and
  restricted que.
\newblock {\em Math. Res. Lett.}, 20(3):465--475, 2013.

\bibitem[DK18]{dasThirdMomentSymmetric2018}
Soumya Das and Rizwanur Khan.
\newblock The third moment of symmetric square {{L}}-functions.
\newblock {\em Quarterly Journal of Mathematics}, 69(3):1063--1087, 2018.

\bibitem[DZ13]{dyatlovQuantumErgodicityRestrictions2012}
Semyon Dyatlov and Maciej Zworski.
\newblock Quantum ergodicity for restrictions to hypersurfaces.
\newblock {\em Nonlinearity}, 26(1):35--52, 2013.

\bibitem[GS12]{ghoshRealZerosHolomorphic2011c}
Amit Ghosh and Peter Sarnak.
\newblock Real zeros of holomorphic {{Hecke}} cusp forms.
\newblock {\em Journal of the European Mathematical Society (JEMS)},
  14(2):465--487, 2012.

\bibitem[HL20]{huangQuantumVarianceDihedral2020}
Bingrong Huang and Stephen Lester.
\newblock Quantum variance for dihedral maass forms.
\newblock 2020.

\bibitem[HS10]{holowinskyMassEquidistributionHecke2010c}
Roman Holowinsky and Kannan Soundararajan.
\newblock Mass equidistribution for {{Hecke}} eigenforms.
\newblock {\em Annals of Mathematics}, 172(2):1517--1528, 2010.

\bibitem[Hua21]{huangQuantumVarianceEisenstein2021}
Bingrong Huang.
\newblock Quantum variance for {{Eisenstein}} series.
\newblock {\em International Mathematics Research Notices. IMRN},
  (2):1224--1248, 2021.

\bibitem[IK04]{IwaniecKowalskiAnalytica}
Henryk Iwaniec and Emmanuel Kowalski.
\newblock {\em Analytic Number Theory}, volume~53 of {\em American Mathematical
  Society Colloquium Publications}.
\newblock {American Mathematical Society, Providence, RI}, 2004.

\bibitem[ILS00]{iwaniecLowLyingZeros2000}
Henryk Iwaniec, Wenzhi Luo, and Peter Sarnak.
\newblock Low lying zeros of families of {{L}}-functions.
\newblock {\em Inst. Hautes \'Etudes Sci. Publ. Math.}, (91):55--131 (2001),
  2000.

\bibitem[Kha10]{khanNonvanishingSymmetricSquare2010}
Rizwanur Khan.
\newblock Non-vanishing of the symmetric square {{L}}-function at the central
  point.
\newblock {\em Proceedings of the London Mathematical Society. Third Series},
  100(3):736--762, 2010.

\bibitem[Lin06]{lindenstraussInvariantMeasuresArithmetic2006}
Elon Lindenstrauss.
\newblock Invariant measures and arithmetic quantum unique ergodicity.
\newblock {\em Annals of Mathematics}, 163(1):165--219, 2006.

\bibitem[LRS09]{luoVarianceArithmeticMeasures2009}
Wenzhi Luo, Ze{\'e}v Rudnick, and Peter Sarnak.
\newblock The variance of arithmetic measures associated to closed geodesics on
  the modular surface.
\newblock {\em Journal of Modern Dynamics}, 3(2):271--309, 2009.

\bibitem[LS95]{luoQuantumErgodicityEigenfunctions1995}
Wen~Zhi Luo and Peter Sarnak.
\newblock Quantum ergodicity of eigenfunctions on
  {{PSL}}{$_2$}({{Z}})\textbackslash{{H}}{$^{2}$}.
\newblock {\em Inst. Hautes \'Etudes Sci. Publ. Math.}, (81):207--237, 1995.

\bibitem[LS03]{luoandsarnakMassEquidistributionHecke}
Wenzhi Luo and Peter Sarnak.
\newblock Mass equidistribution for {{Hecke}} eigenforms, 2003.

\bibitem[LS04]{luoQuantumVarianceHecke2004a}
Wenzhi Luo and Peter Sarnak.
\newblock Quantum variance for {{Hecke}} eigenforms.
\newblock {\em Ann. Sci. \'Ecole Norm. Sup. (4)}, 37(5):769--799, 2004.

\bibitem[Nel16]{nelsonQuantumVarianceQuaternion2016}
Paul~D. Nelson.
\newblock Quantum variance on quaternion algebras, {{I}}.
\newblock {\em arXiv:1601.02526 [math]}, November 2016.

\bibitem[Nel17]{nelsonQuantumVarianceQuaternion2017}
Paul~D. Nelson.
\newblock Quantum variance on quaternion algebras, {{II}}.
\newblock {\em arXiv:1702.02669 [math]}, February 2017.

\bibitem[Nel19]{nelsonQuantumVarianceQuaternion2019}
Paul~D. Nelson.
\newblock Quantum variance on quaternion algebras, {{III}}.
\newblock {\em arXiv:1903.08686 [math]}, March 2019.

\bibitem[NPR21]{nordentoftSmallScaleEquidistribution2021b}
Asbjorn~Christian Nordentoft, Yiannis~N. Petridis, and Morten~S. Risager.
\newblock Small scale equidistribution of {{Hecke}} eigenforms at infinity.
\newblock {\em arXiv:2011.05810 [math]}, March 2021.

\bibitem[RS94]{rudnickBehaviourEigenstatesArithmetic1994}
Ze{\'e}v Rudnick and Peter Sarnak.
\newblock The behaviour of eigenstates of arithmetic hyperbolic manifolds.
\newblock {\em Communications in Mathematical Physics}, 161(1):195--213, 1994.

\bibitem[Sou10]{soundararajanQuantumUniqueErgodicity2010}
Kannan Soundararajan.
\newblock Quantum unique ergodicity for {{SL}}{$_2$}({{Z}})\textbackslash{{H}}.
\newblock {\em Annals of Mathematics}, 172(2):1529--1538, 2010.

\bibitem[SZ19]{sarnakQuantumVarianceModular2013}
P.~Sarnak and P.~Zhao.
\newblock The quantum variance of the modular surface.
\newblock 52(5):1155--1200, 2019.

\bibitem[TZ13]{tothQuantumErgodicRestriction2011}
John~A. Toth and Steve Zelditch.
\newblock Quantum ergodic restriction theorems: manifolds without boundary.
\newblock {\em Geom. Funct. Anal.}, 23(2):715--775, 2013.

\bibitem[You16]{youngQuantumUniqueErgodicity2016}
Matthew~P. Young.
\newblock The quantum unique ergodicity conjecture for thin sets.
\newblock {\em Advances in Mathematics}, 286:958--1016, 2016.

\bibitem[Zel94]{zelditchRateQuantumErgodicity1994}
Steven Zelditch.
\newblock On the rate of quantum ergodicity. {{I}}. {{Upper}} bounds.
\newblock {\em Communications in Mathematical Physics}, 160(1):81--92, 1994.

\bibitem[Zha10]{zhaoQuantumVarianceMaassHecke2010}
Peng Zhao.
\newblock Quantum variance of {{Maass}}-{{Hecke}} cusp forms.
\newblock {\em Communications in Mathematical Physics}, 297(2):475--514, 2010.

\end{thebibliography}

\end{document}